\documentclass[twoside,11pt]{amsart}
\usepackage{graphicx, lscape, rotating}
\usepackage{amsmath, amsfonts, amssymb, ifthen}
\usepackage{xypic}

\pdfoutput=1

\newcommand{\field}[1]{\mathbb{#1}}
\newcommand{\CC}{\ensuremath{\field{C}}} 
\newcommand{\Sbb}{\ensuremath{\field{S}}} 
\newcommand{\MM}{\ensuremath{\field{M}}} 
\newcommand{\RR}{\ensuremath{\field{R}}} 
 
\newcommand{\DD}{\ensuremath{\field{D}}} 
\newcommand{\ZZ}{\ensuremath{\field{Z}}} 

\newcommand{\Ch}{\ensuremath{\hat{\field{C}}}}

\newcommand{\Sone}{\ensuremath{\field{S}^1}}
\newcommand{\Stwo}{\ensuremath{\field{S}^2}}
\newcommand{\Fam}{\ensuremath{\mathcal{F}}}
\newcommand{\Mspace}{\ensuremath{\mathcal{M}}}

\newcommand{\mate}{$ \begin{turn}{90}$\vDash$\end{turn} $}

\theoremstyle{plain}
\newtheorem{thm}{Theorem}[section]

\newtheorem{lem}[thm]{Lemma}
\newtheorem{fact}[thm]{Fact}

\theoremstyle{definition}
\newtheorem{defn}[thm]{Definition}
\newtheorem{notation}[thm]{Notation}
\newcommand{\drawfigJonesixth}{\scalebox{.4}{\includegraphics{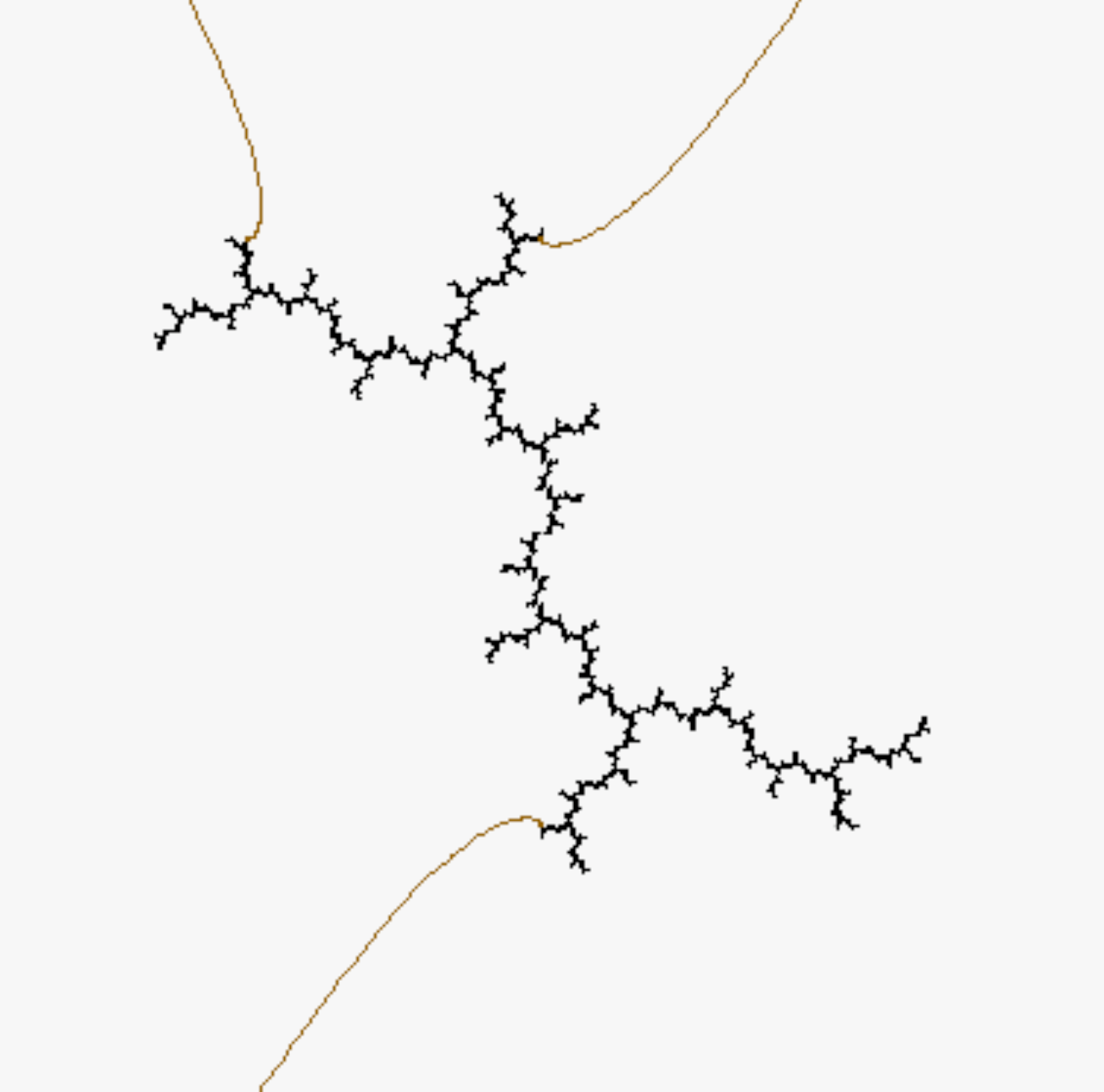}}}
\newcommand{\drawfigJonefourth}{\scalebox{.4}{\includegraphics{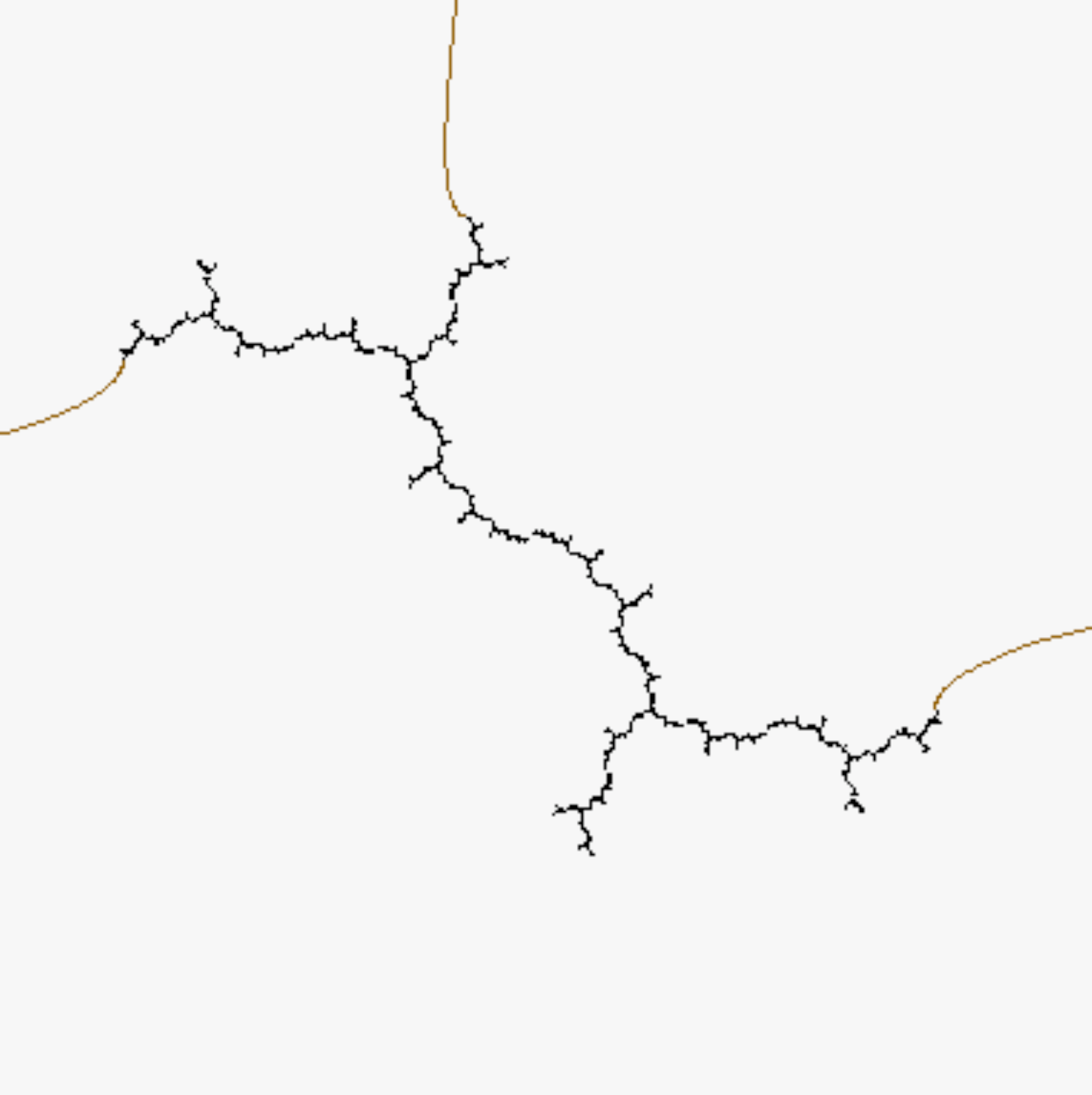}}}

\newcommand{\drawfigJonethirdonesixth}{\scalebox{.4}{\includegraphics{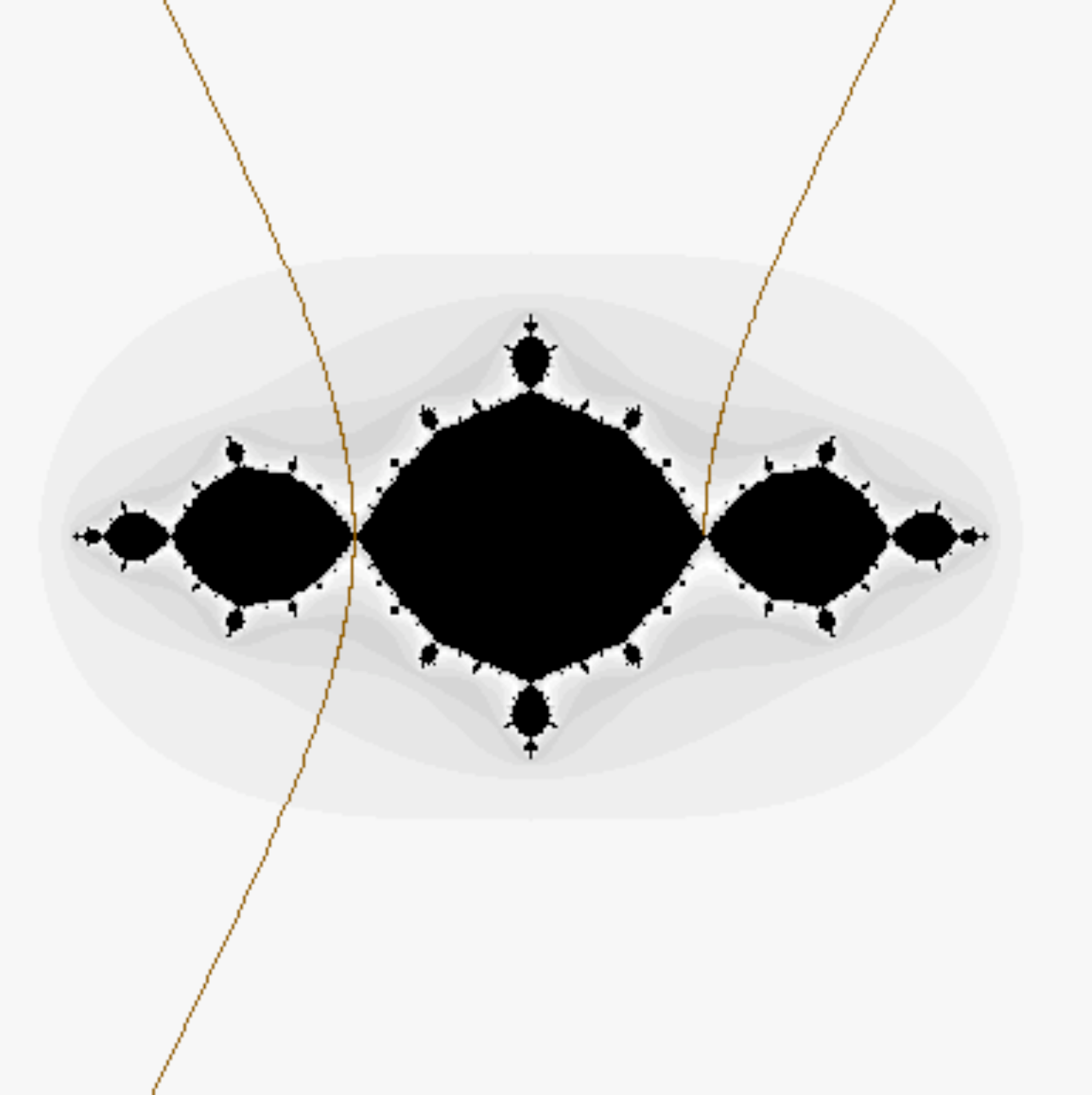}}}

\newcommand{\drawfigJonethirdoneseventh}{\scalebox{.3}{\includegraphics{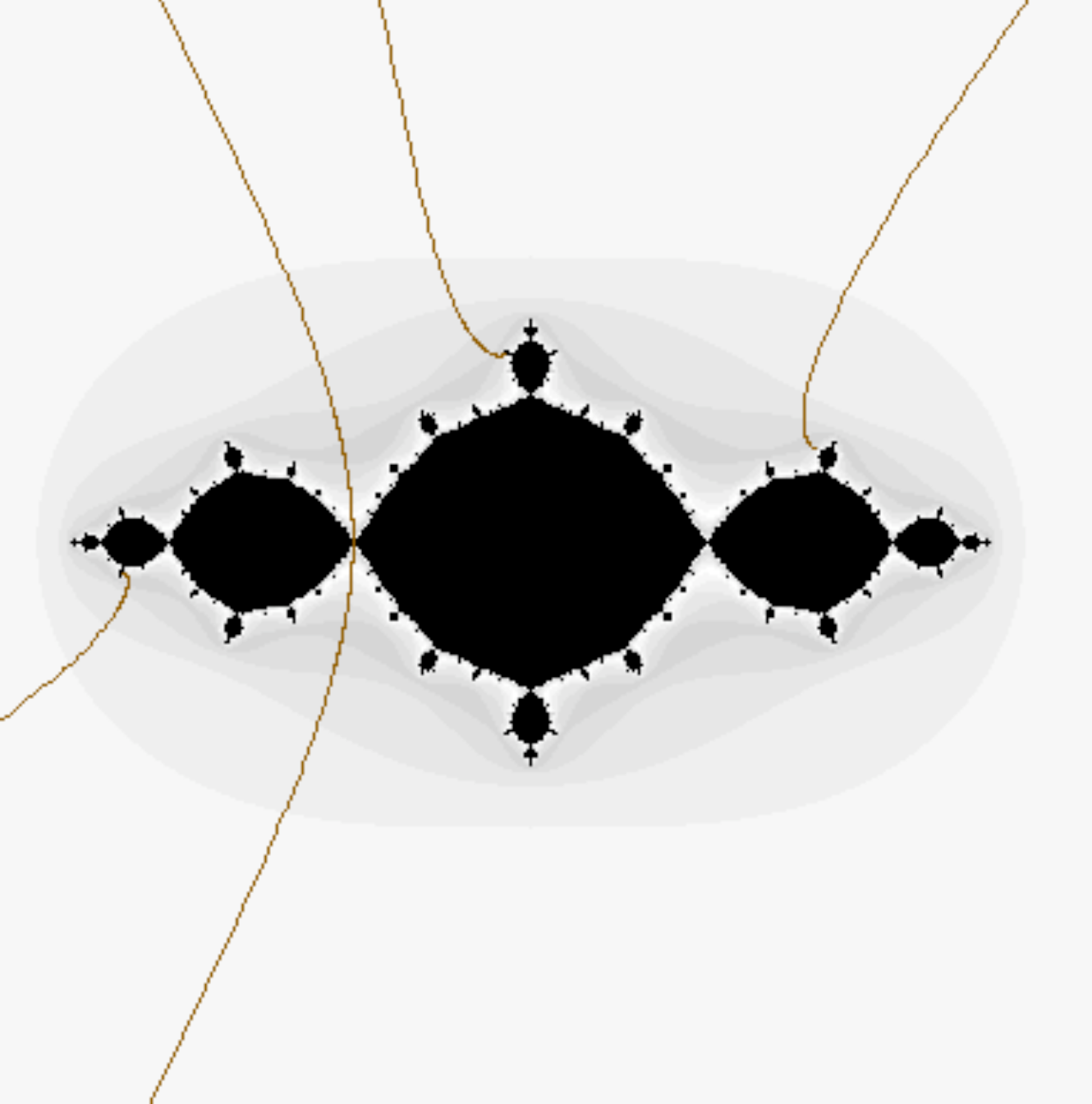}}}
\newcommand{\drawfigJoneseventhonethird}{\scalebox{.3}{\includegraphics{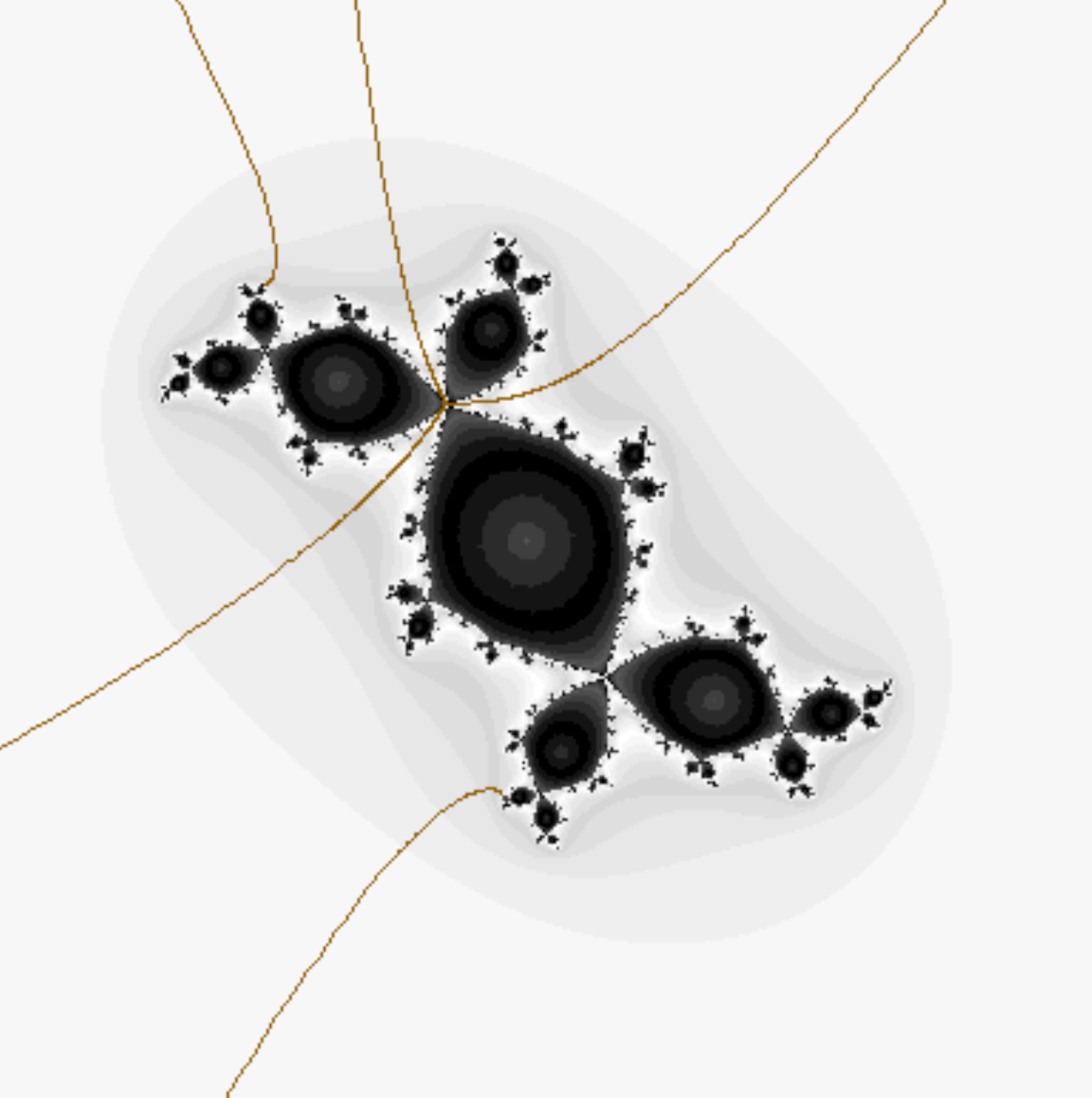}}}

\newcommand{\drawfigJonesevenththreeseventh}{\scalebox{.4}{\includegraphics{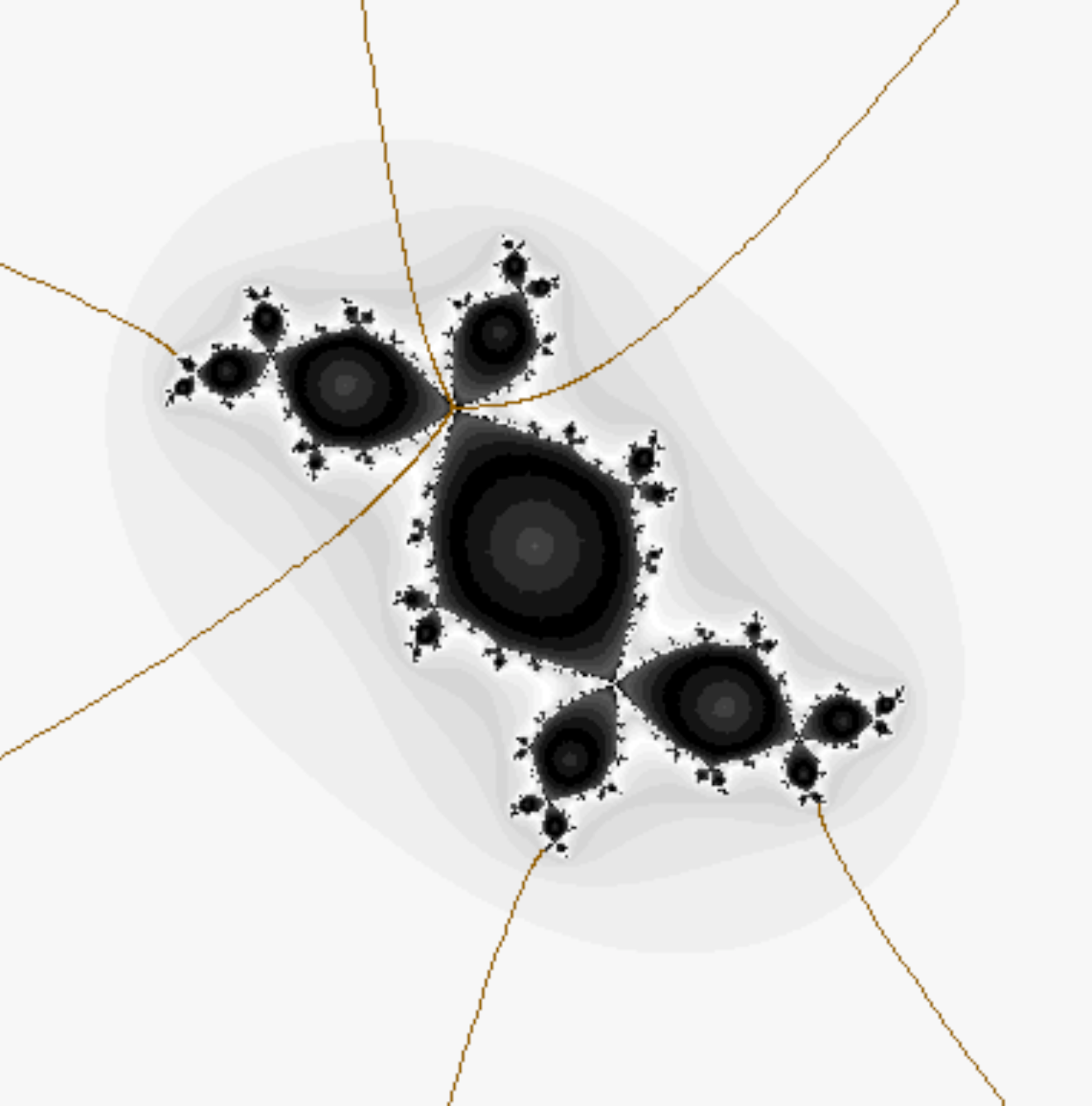}}}
\newcommand{\drawfigJthreeseventhoneseventh}{\scalebox{.4}{\includegraphics{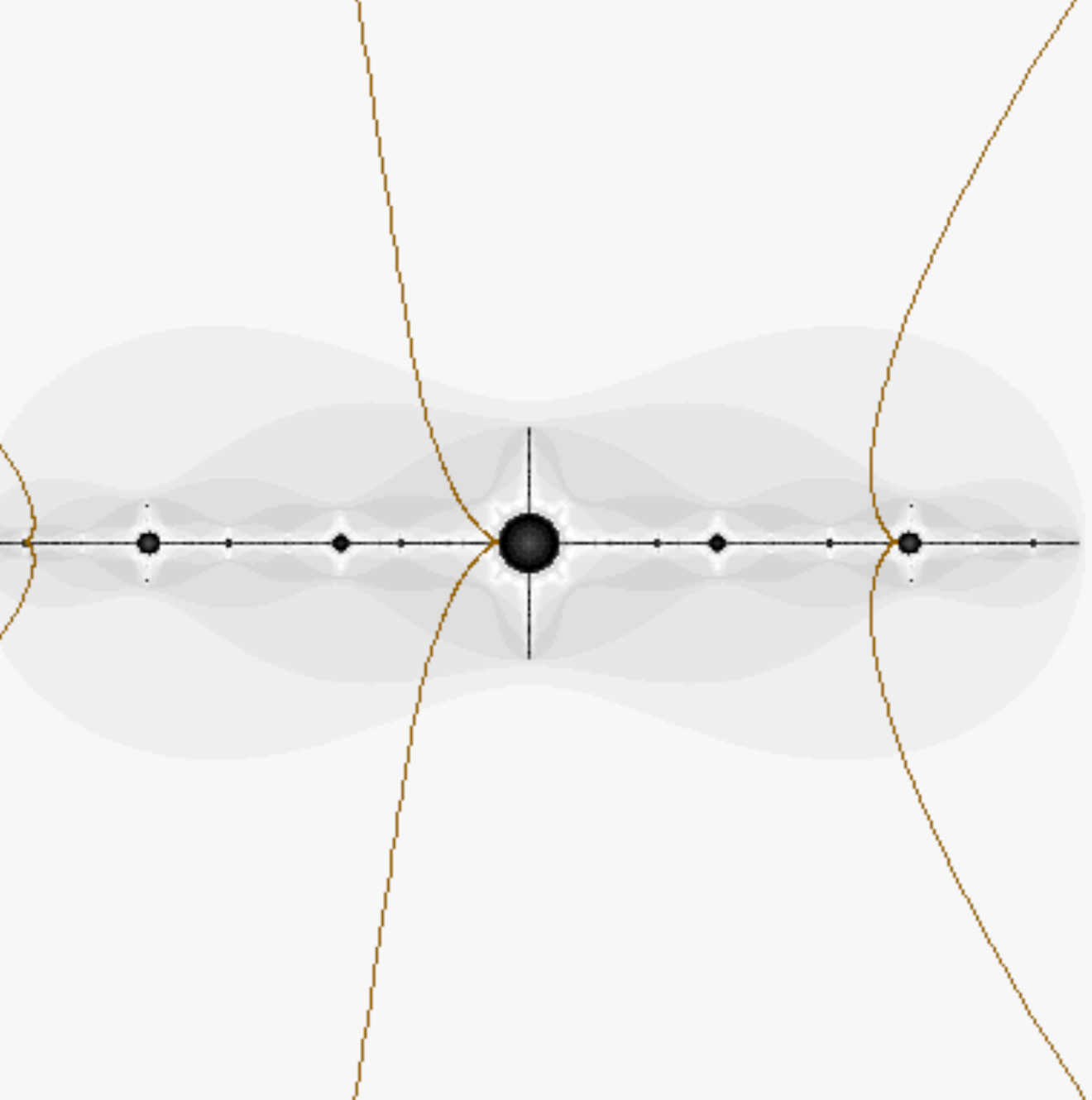}}}

\newcommand{\drawfigJonefifth}{\scalebox{.4}{\includegraphics{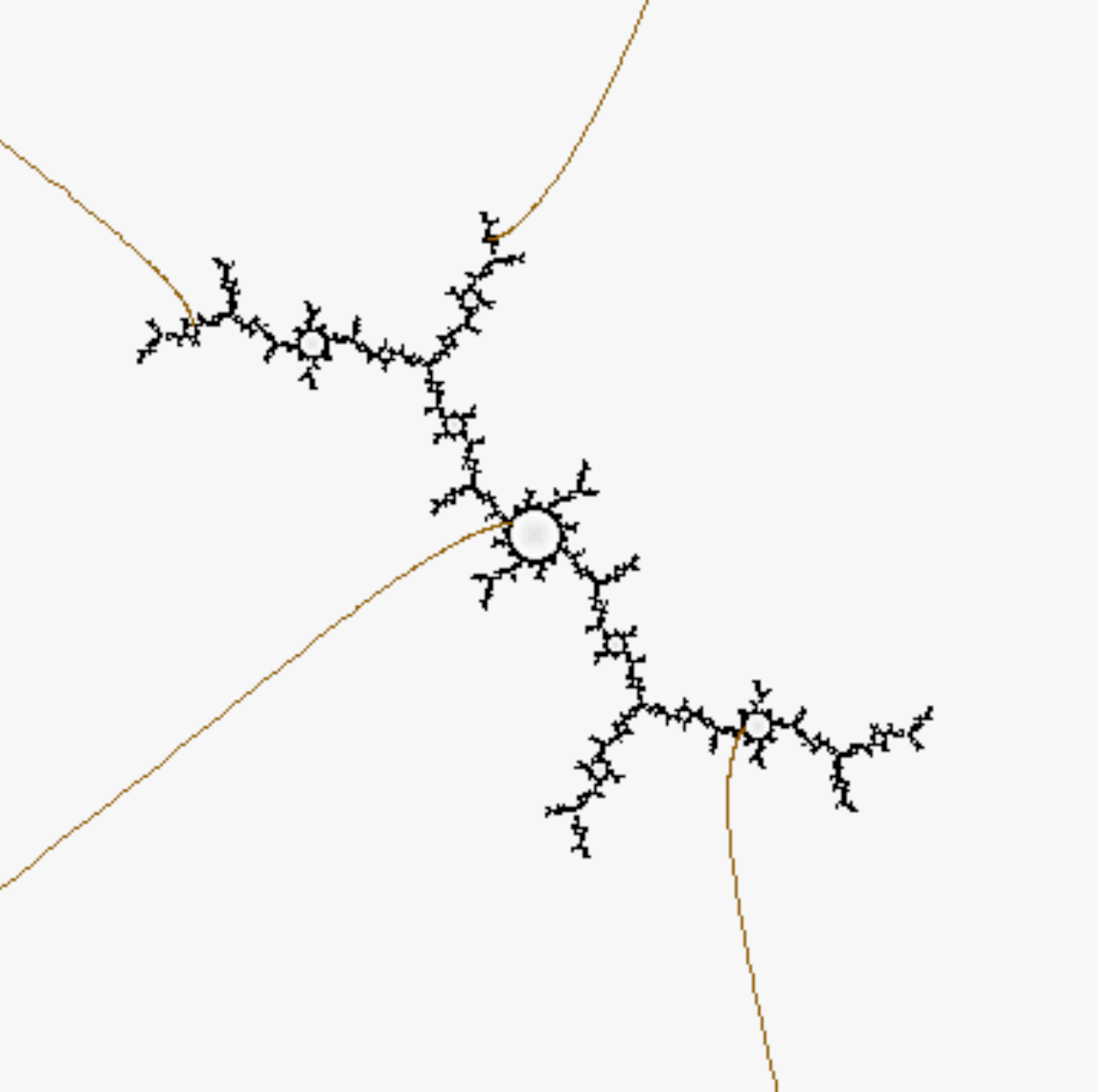}}}

\newcommand{\drawfigmandel}{\scalebox{.8}{\includegraphics{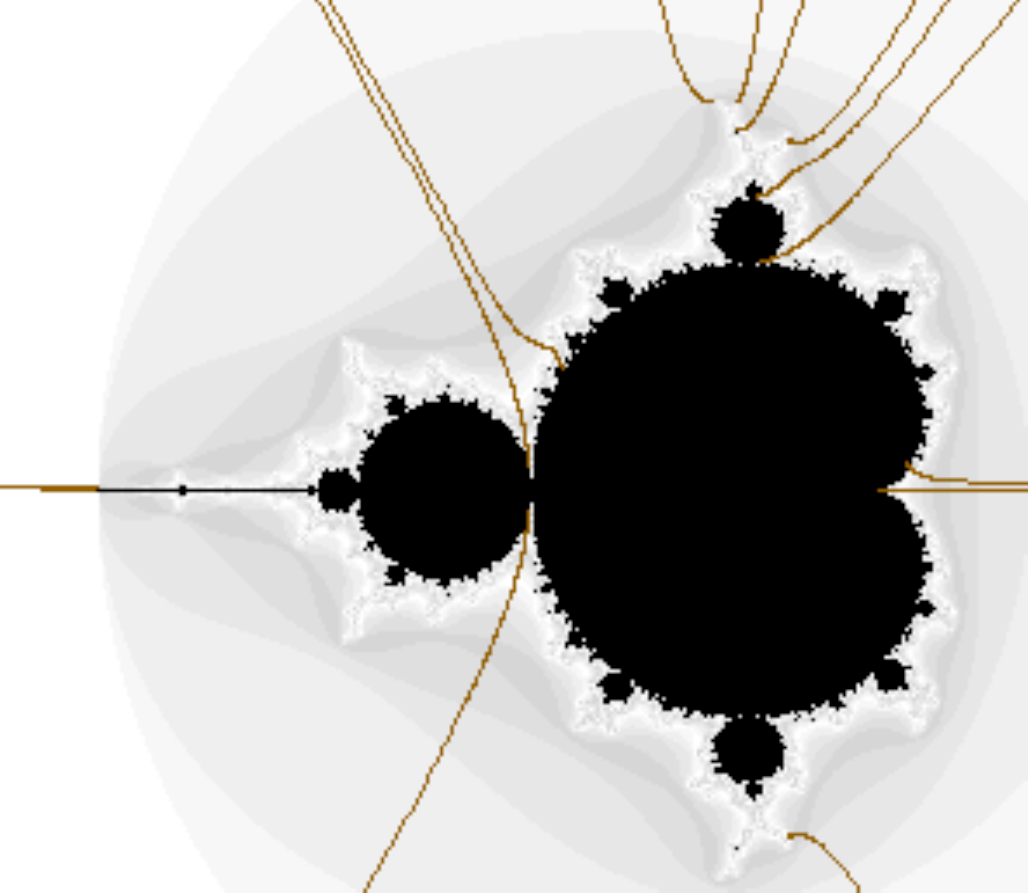}}}

\newcommand{\drawfigspiders}{\scalebox{.175}{\includegraphics{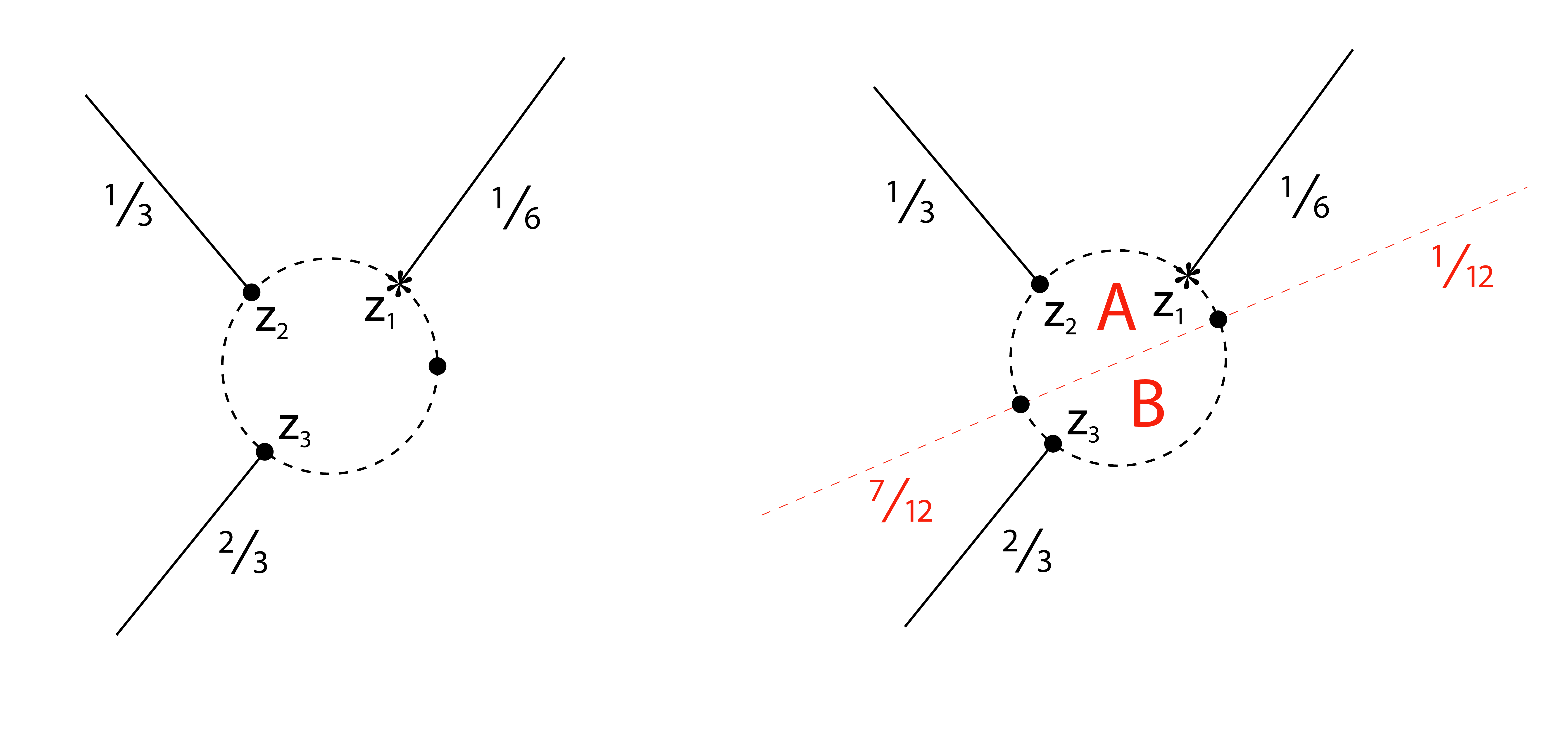}}}
\newcommand{\drawfigformmedusa}{\scalebox{.17}{\includegraphics{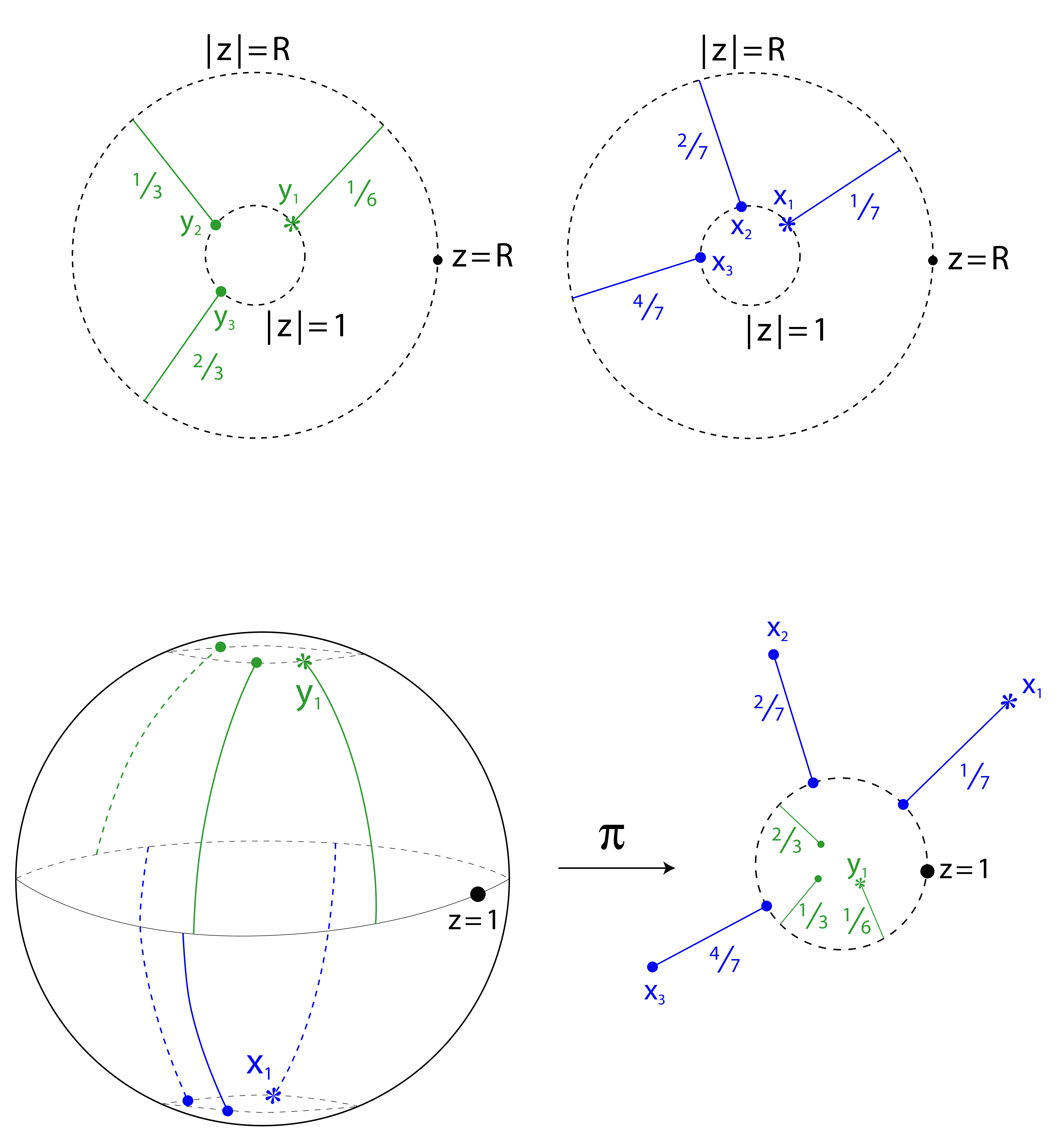}}}

\newcommand{\drawfigrealmedusazero}{\scalebox{.1}{\includegraphics{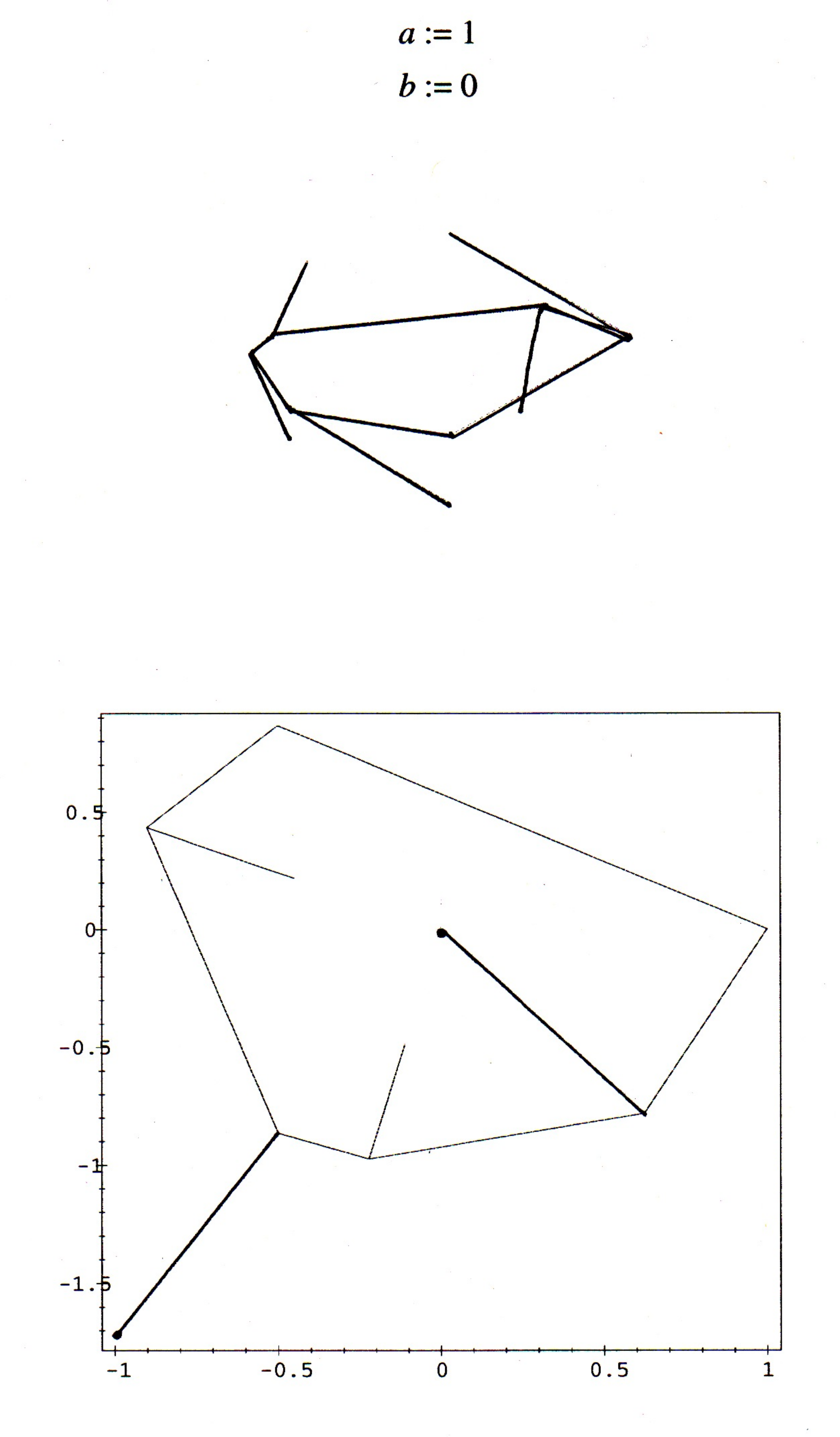}}}
\newcommand{\drawfigrealmedusatwo}{\scalebox{.1}{\includegraphics{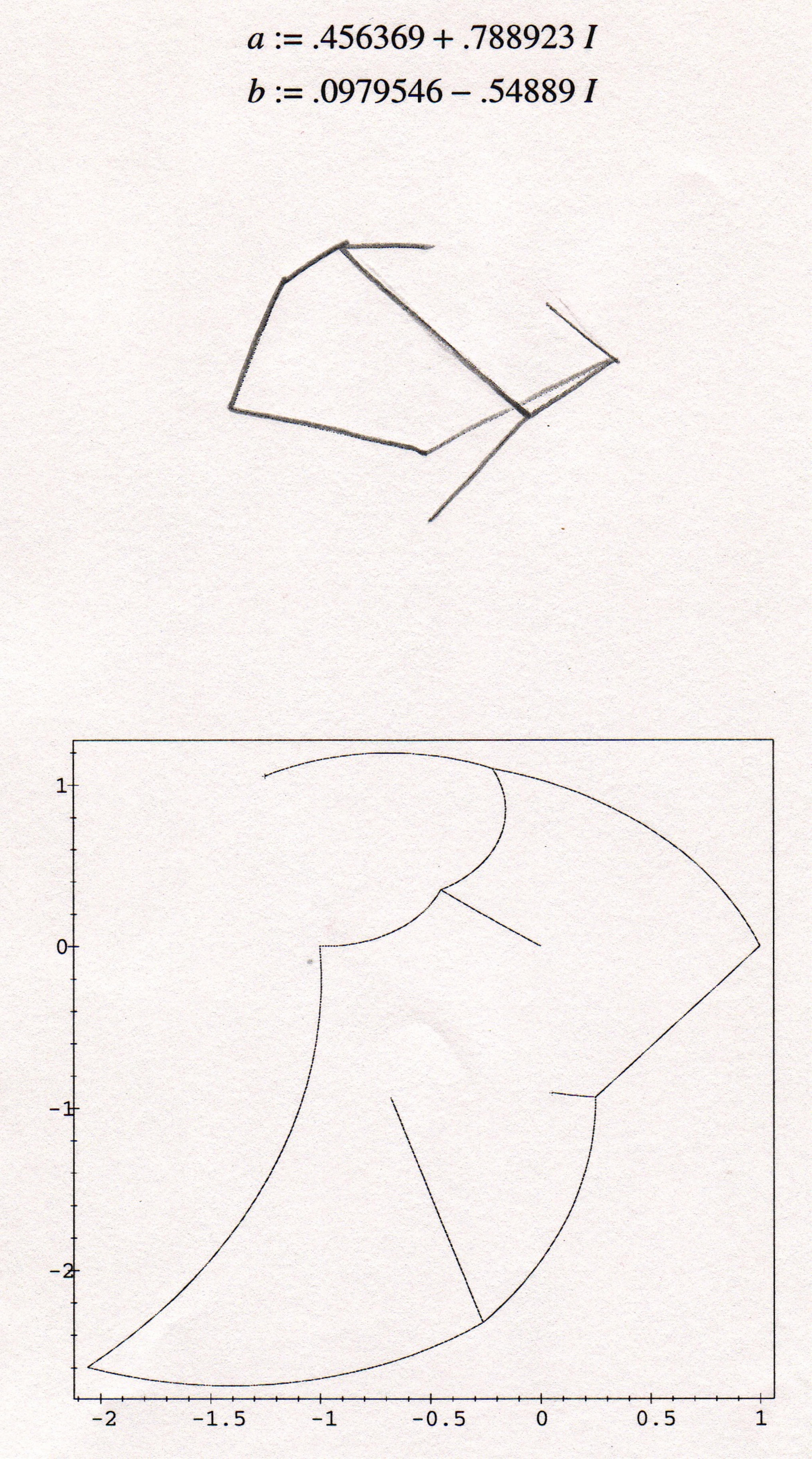}}}
\newcommand{\drawfigrealmedusatwenty}{\scalebox{.1}{\includegraphics{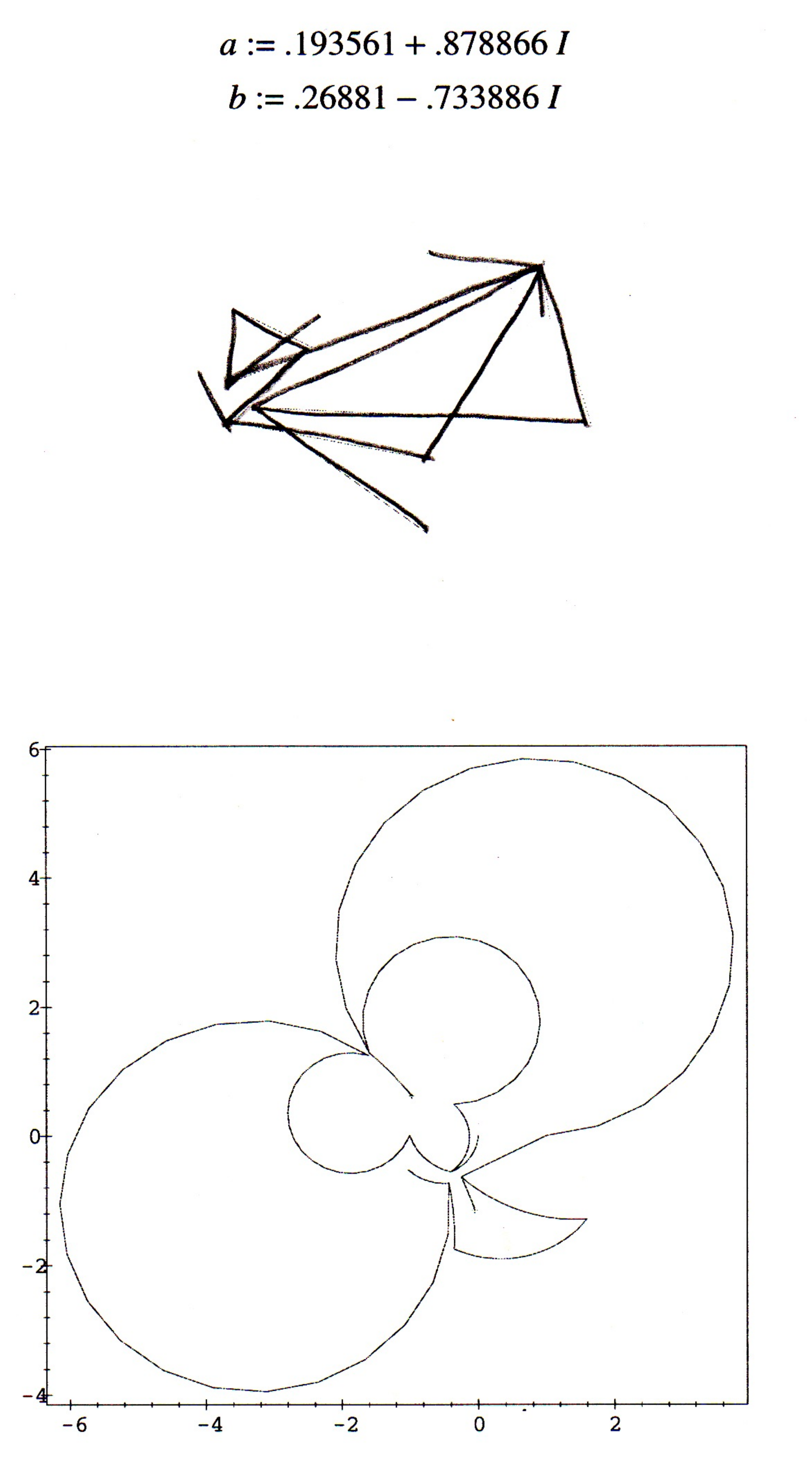}}}

\newcommand{\drawfiglatone}{\scalebox{.25}{\includegraphics{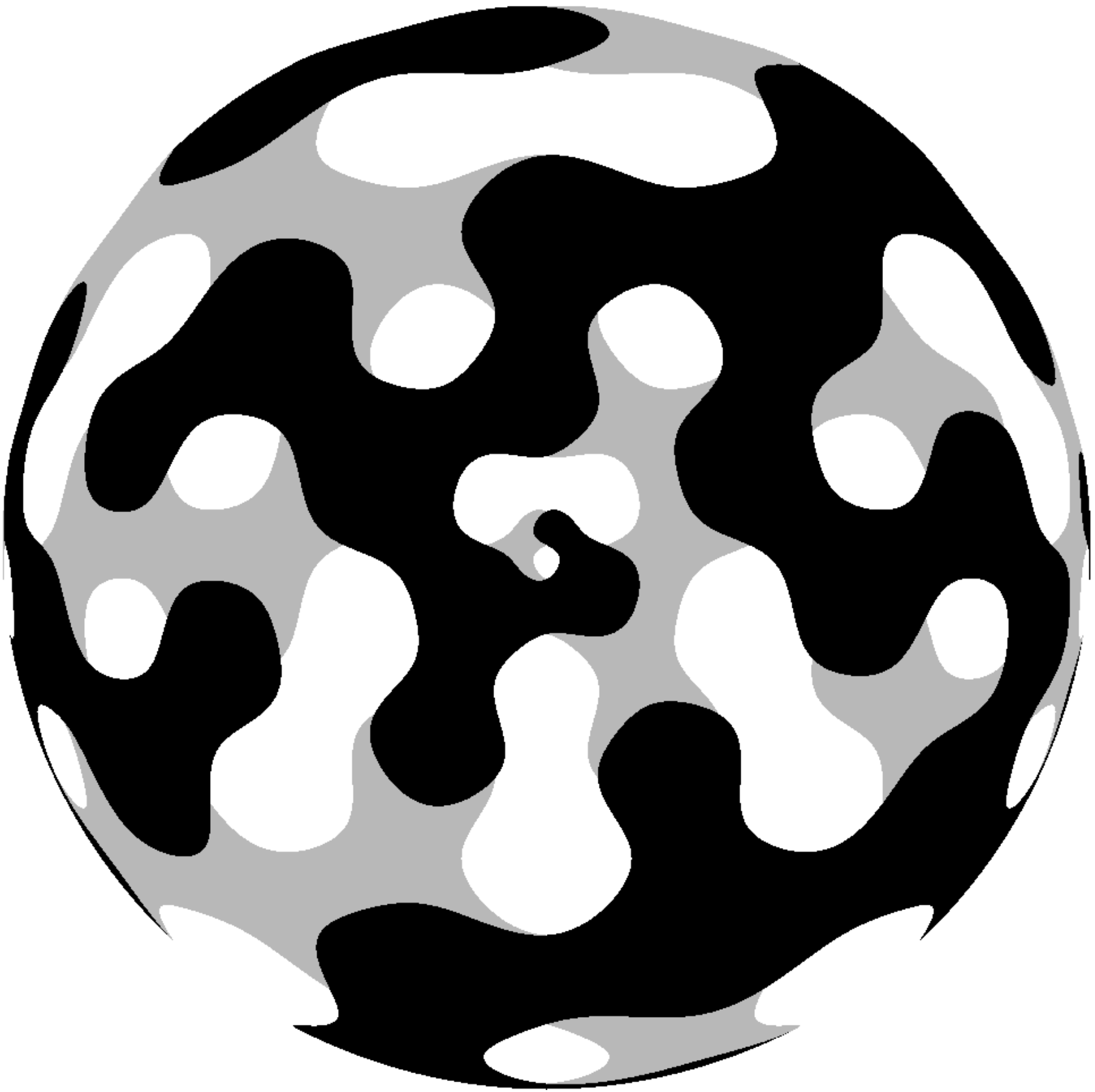}}}

\newcommand{\drawfiglattwo}{\scalebox{.25}{\includegraphics{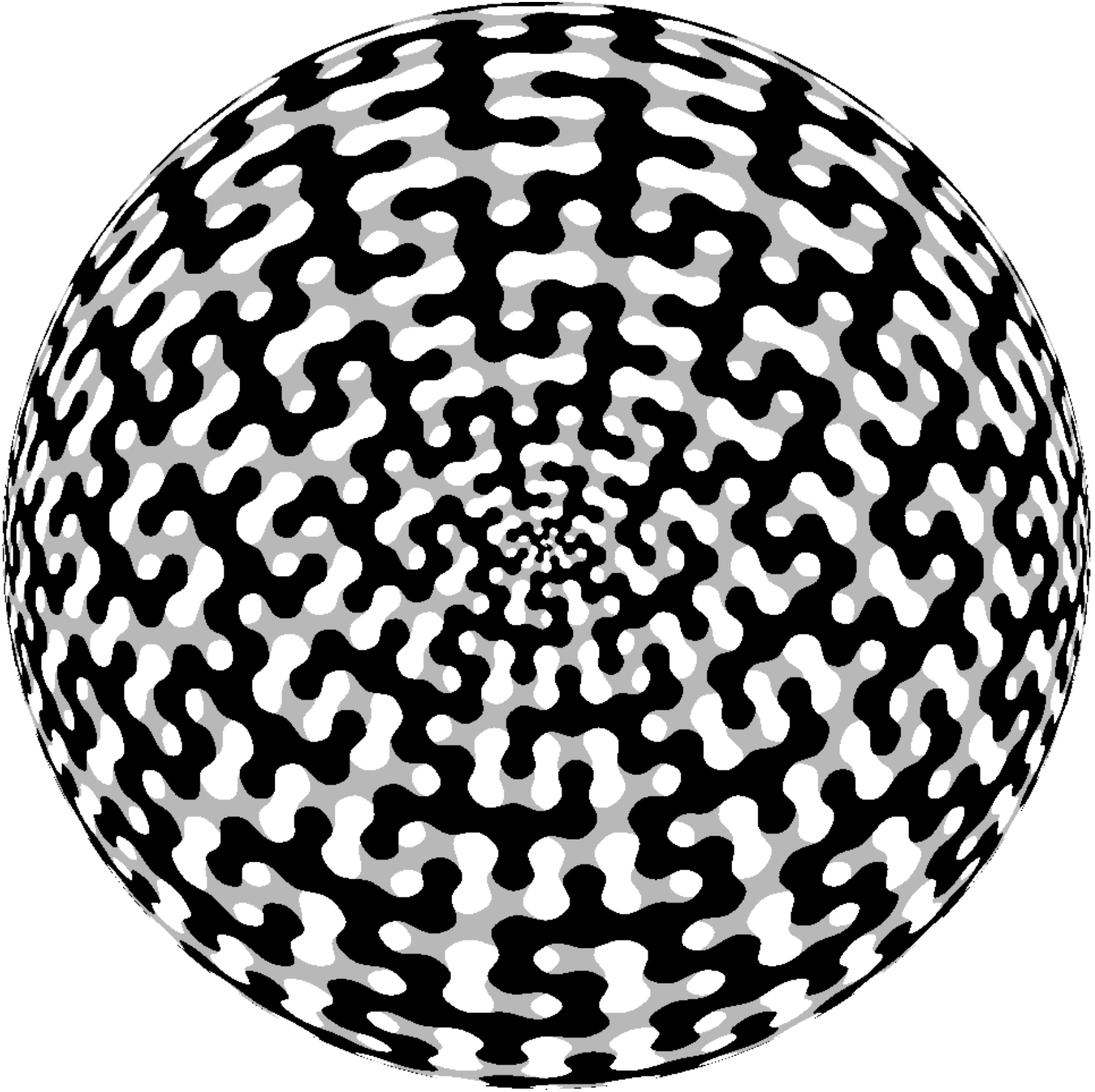}}}

\newcommand{\drawfiglatthree}{\scalebox{.25}{\includegraphics{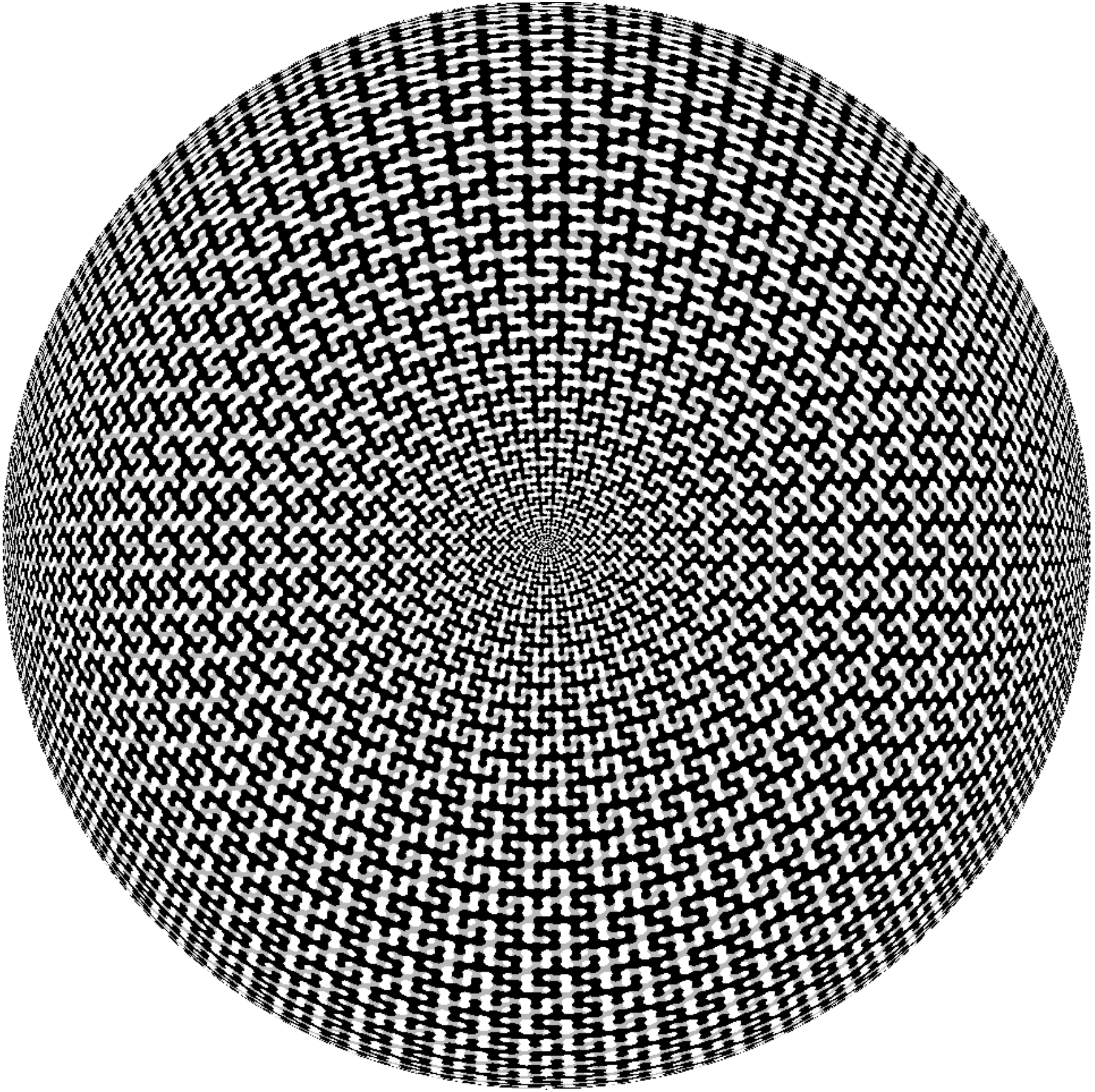}}}

\newcommand{\drawfigrabmatebas}{\scalebox{.25}{\includegraphics{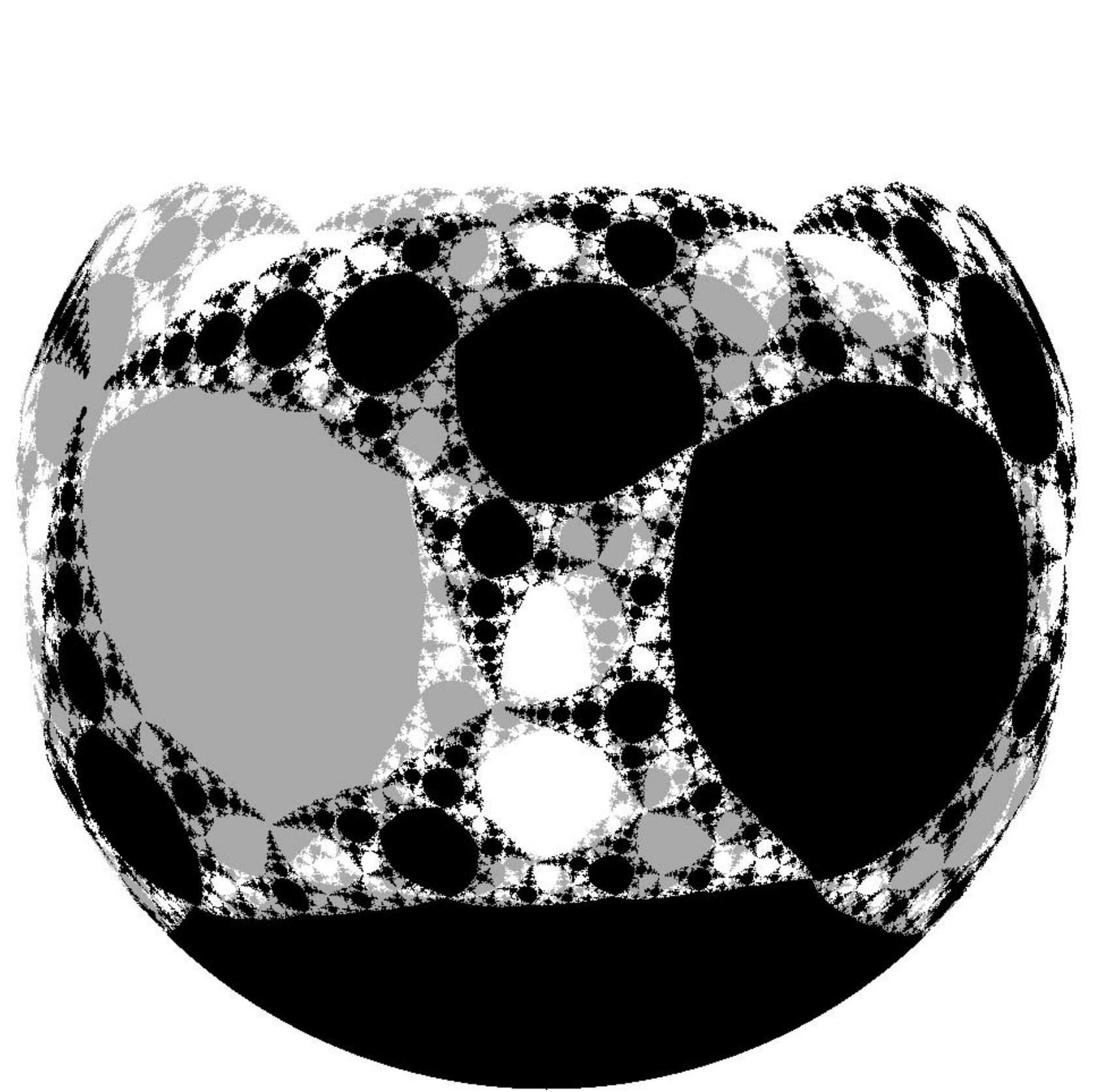}}}

\newcommand{\drawfigrabmaterab}{\scalebox{.25}{\includegraphics{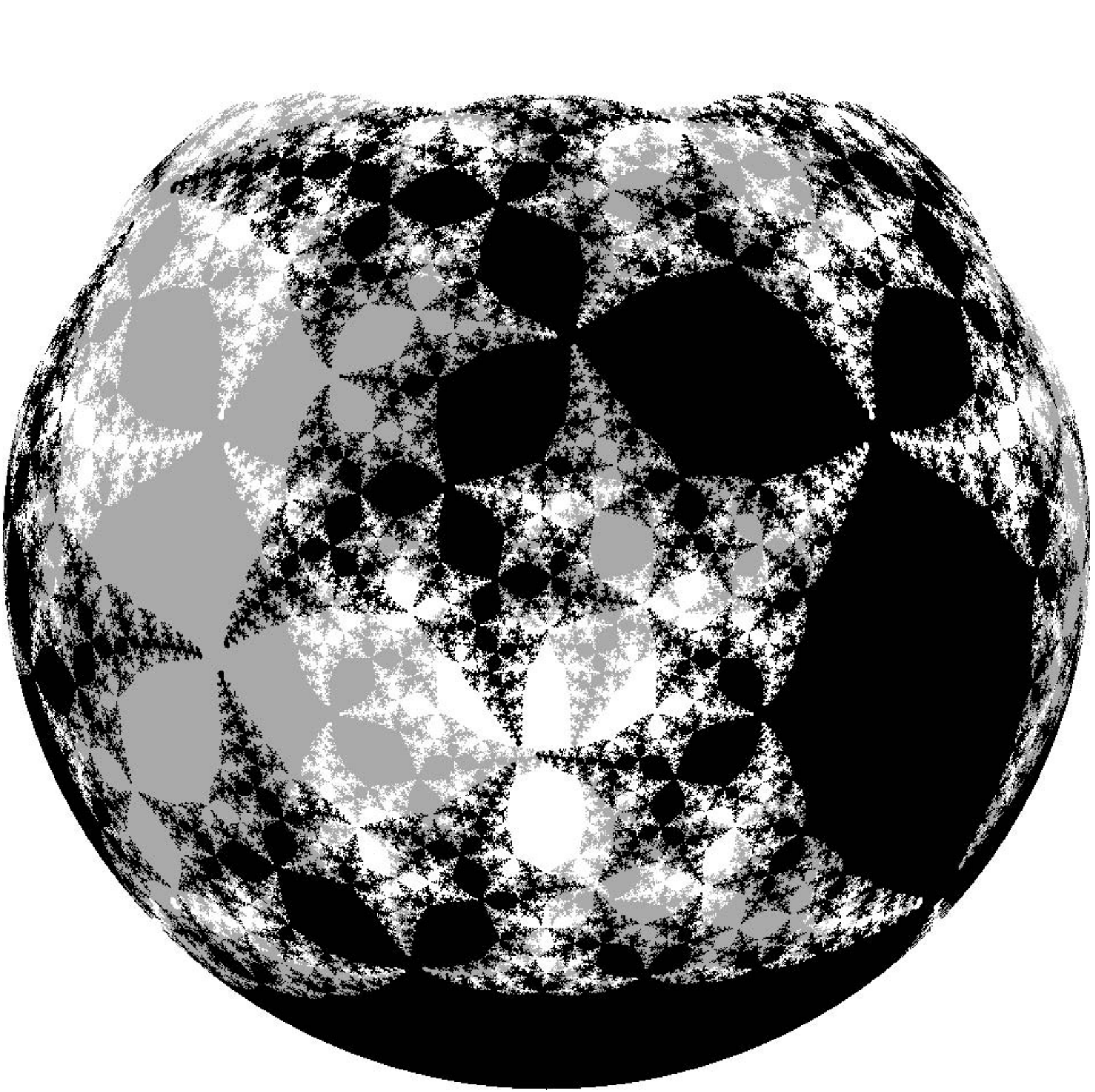}}}

 \newcommand{\drawfigrabMbasinrab}{\scalebox{.25}{\includegraphics{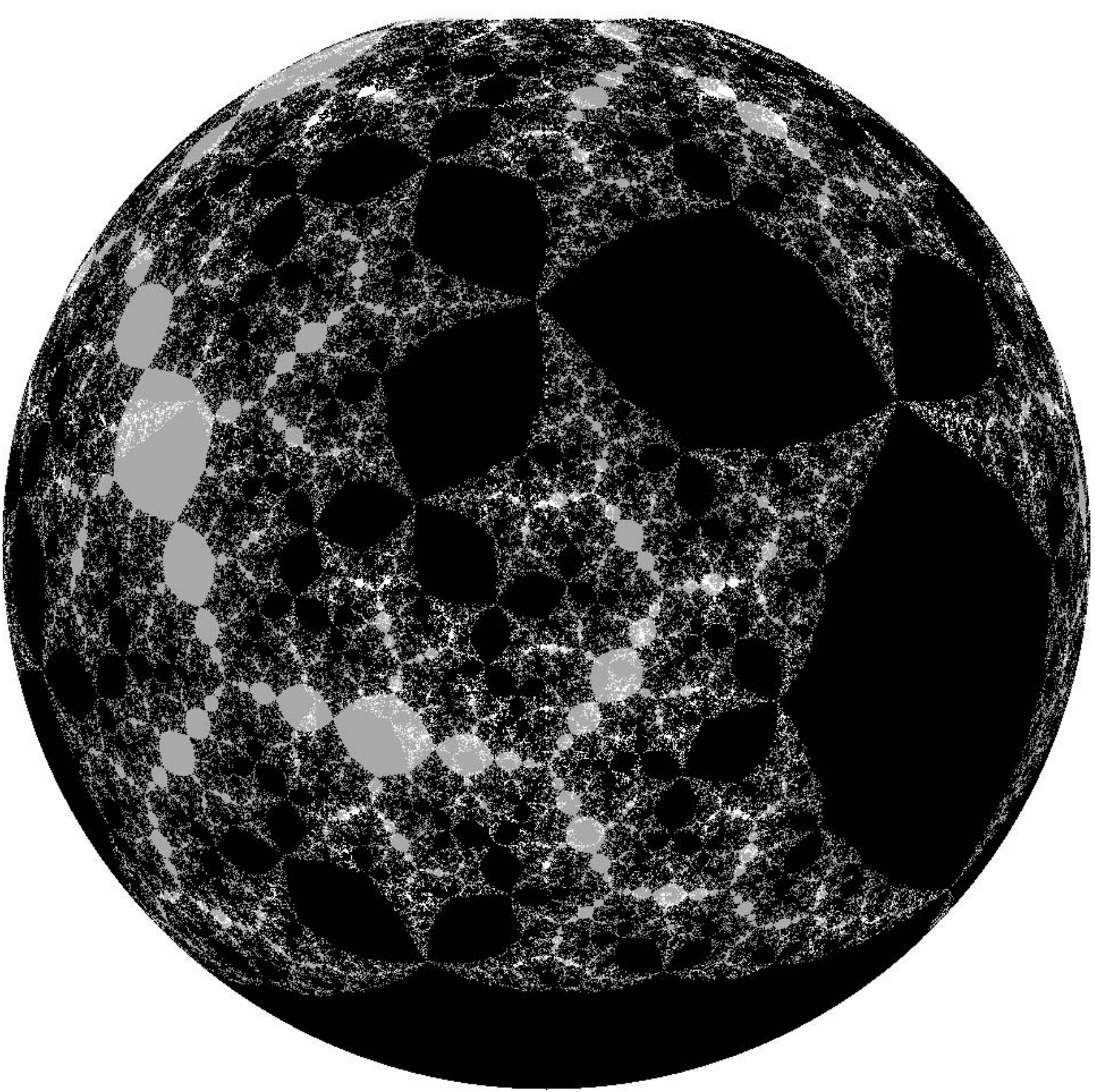}}}

\newcommand{\drawfigrabmateaero}{\scalebox{.25}{\includegraphics{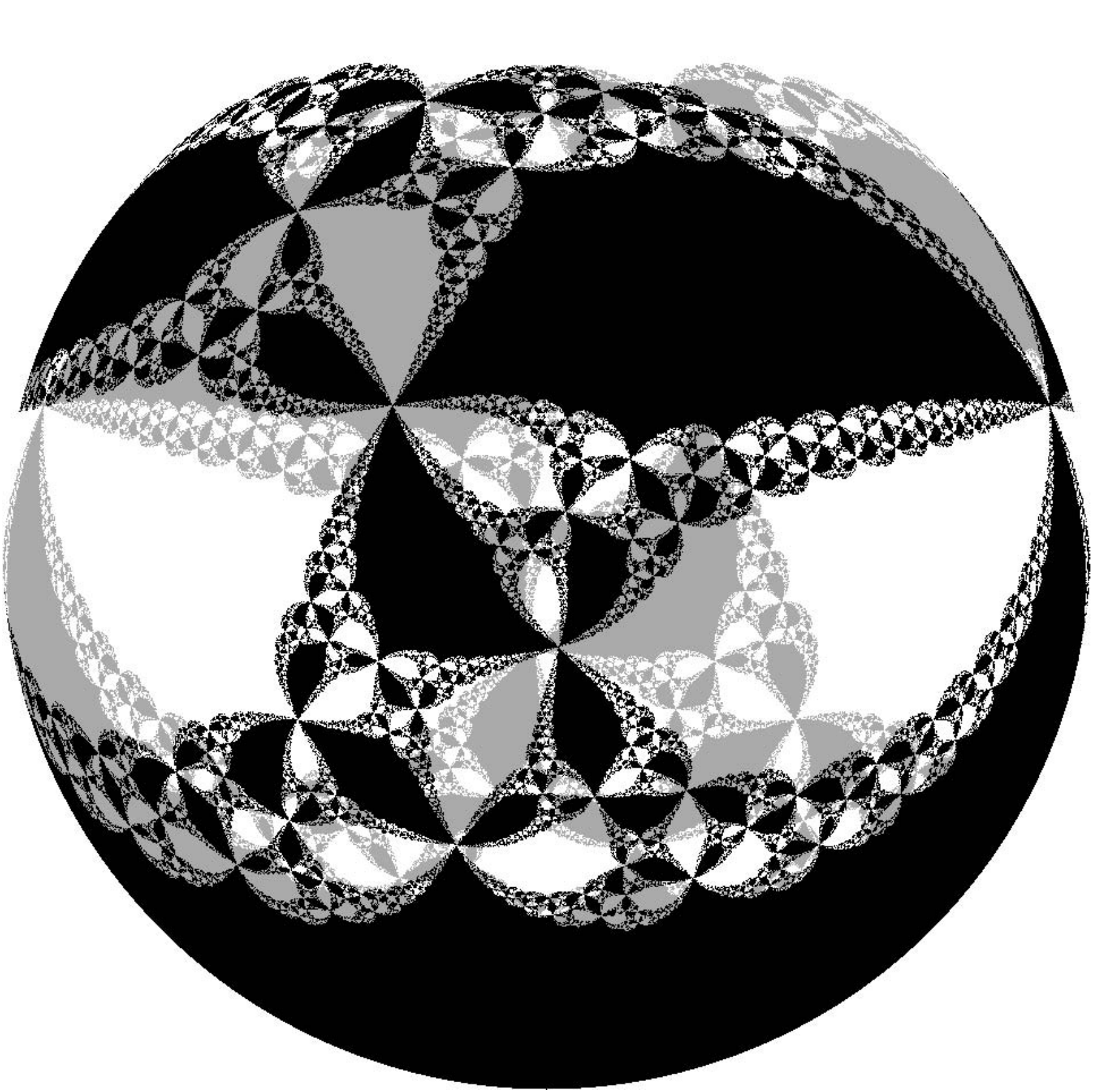}}}

\newcommand{\drawfigbasinrabmateaero}{\scalebox{.25}{\includegraphics{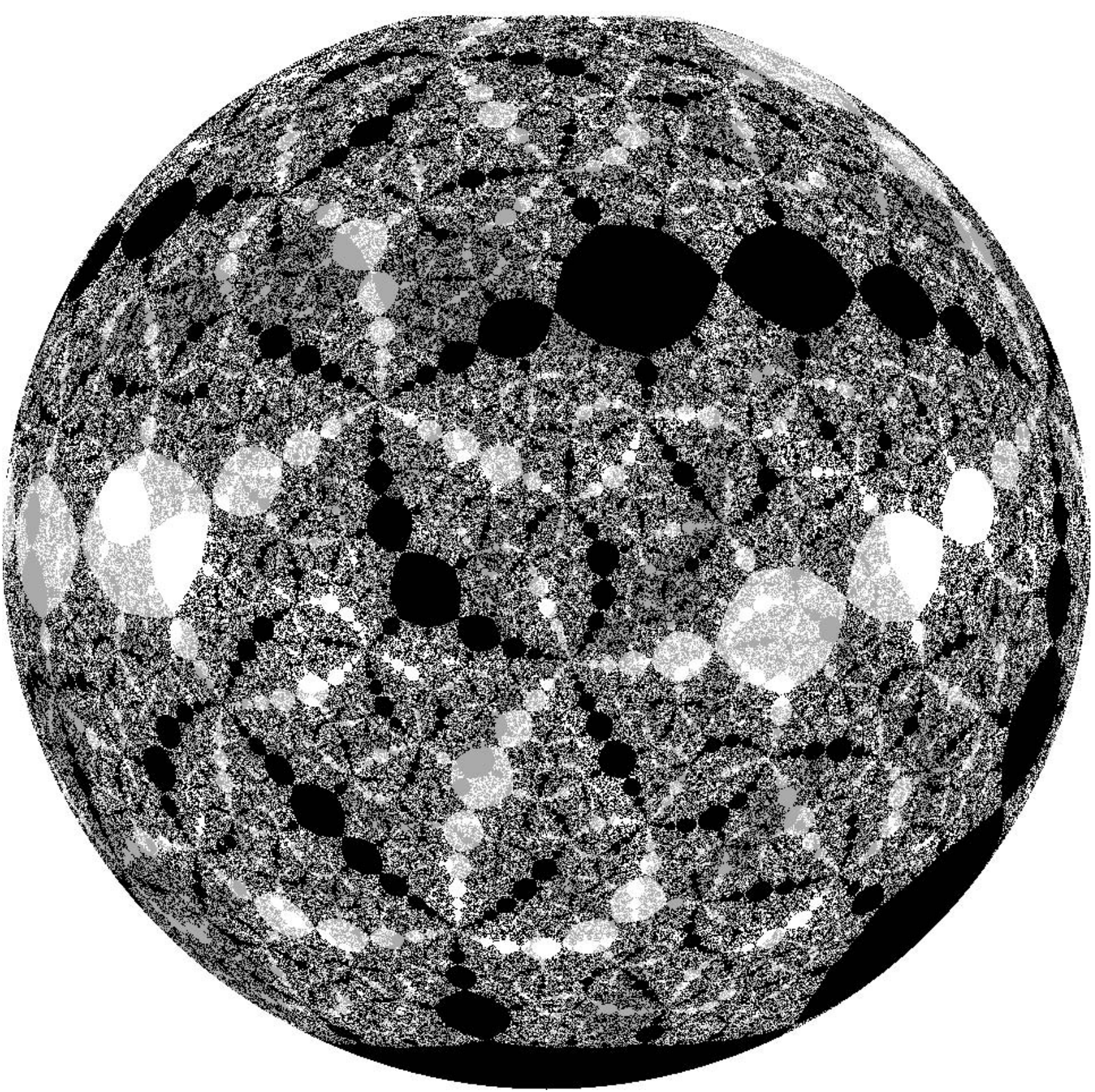}}}

\newcommand{\drawfiglattesone}{\scalebox{.25}{\includegraphics{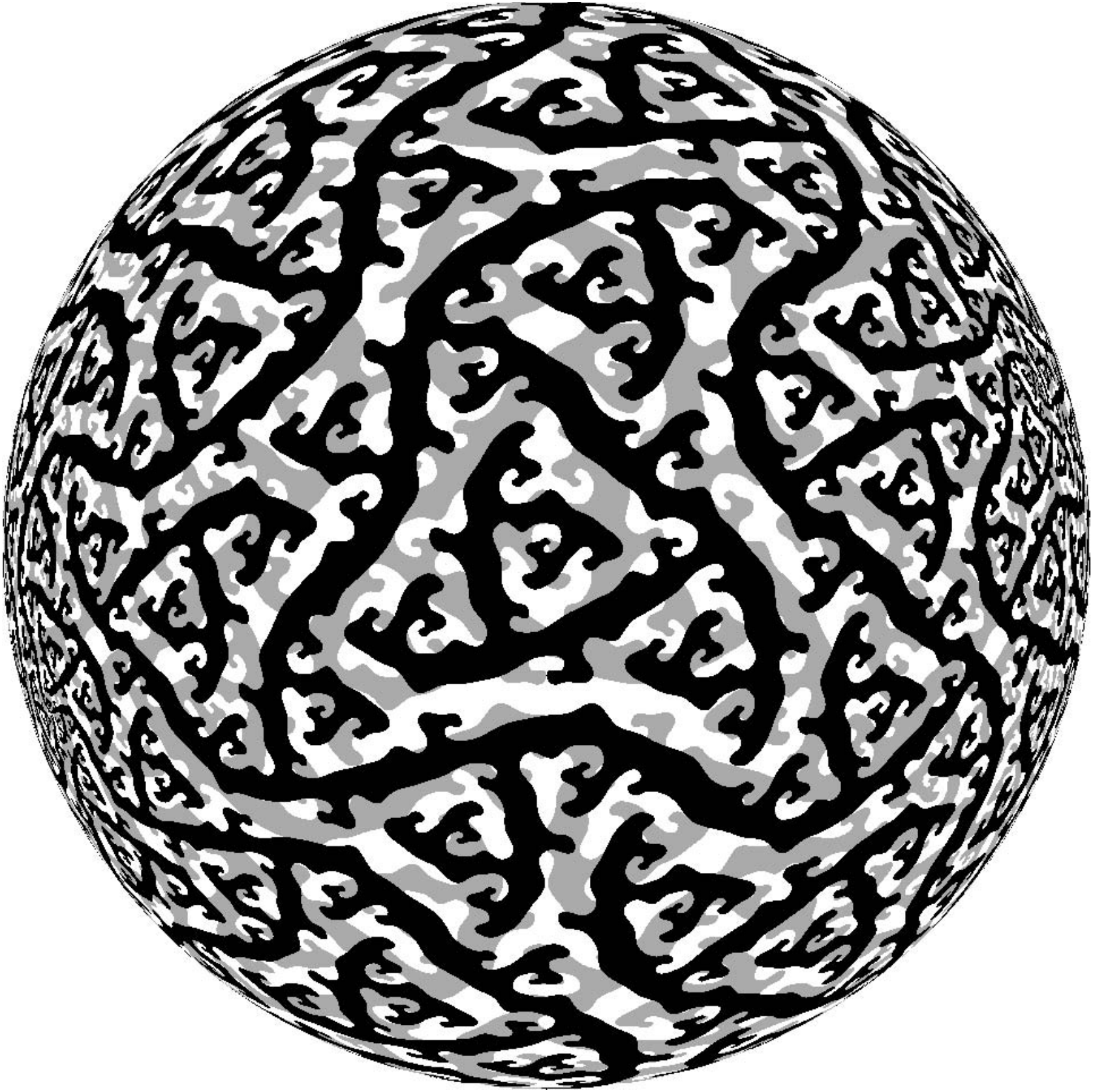}}}
\newcommand{\drawfiglattestwo}{\scalebox{.25}{\includegraphics{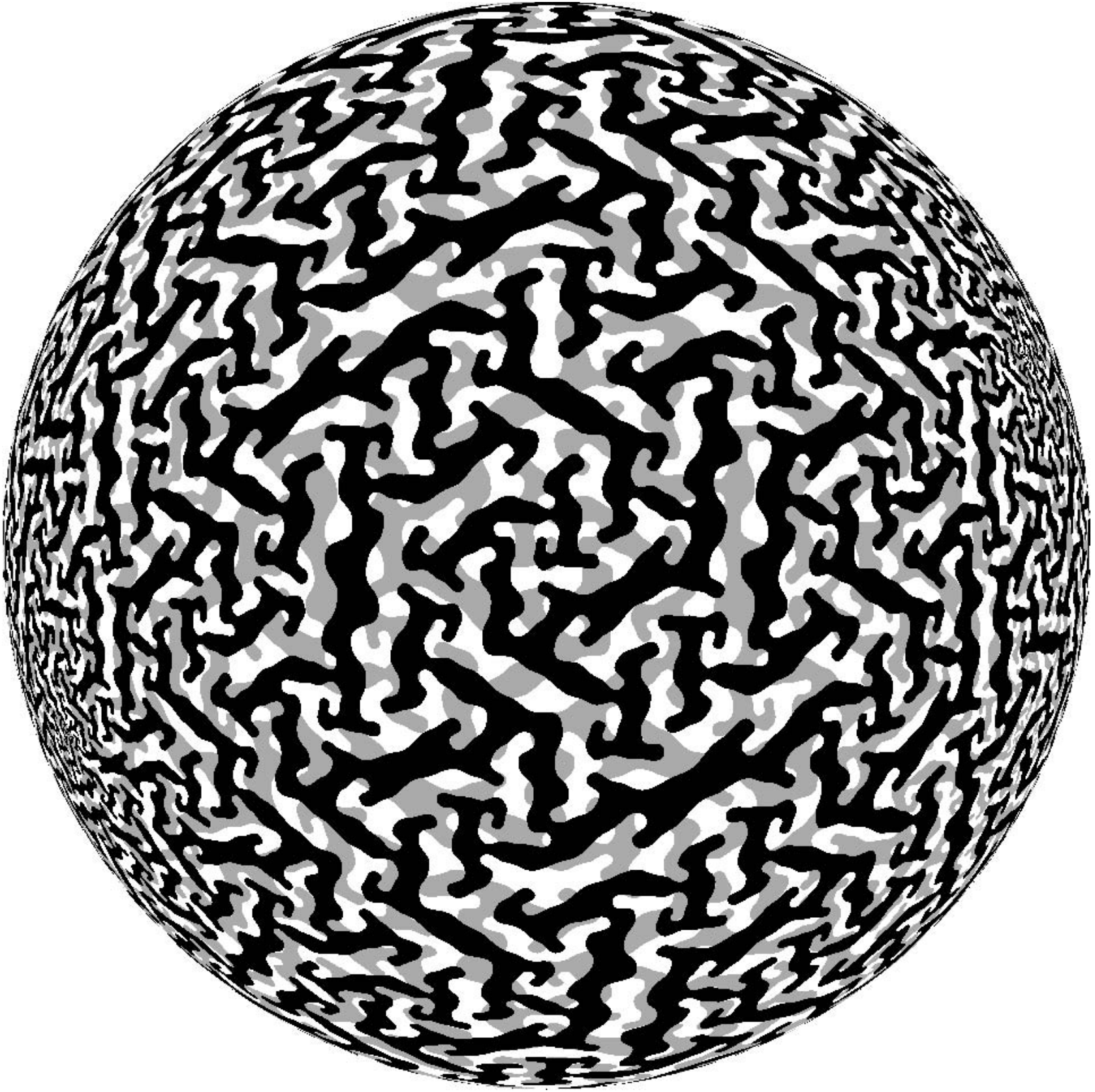}}}
\newcommand{\drawfiglattesthree}{\scalebox{.25}{\includegraphics{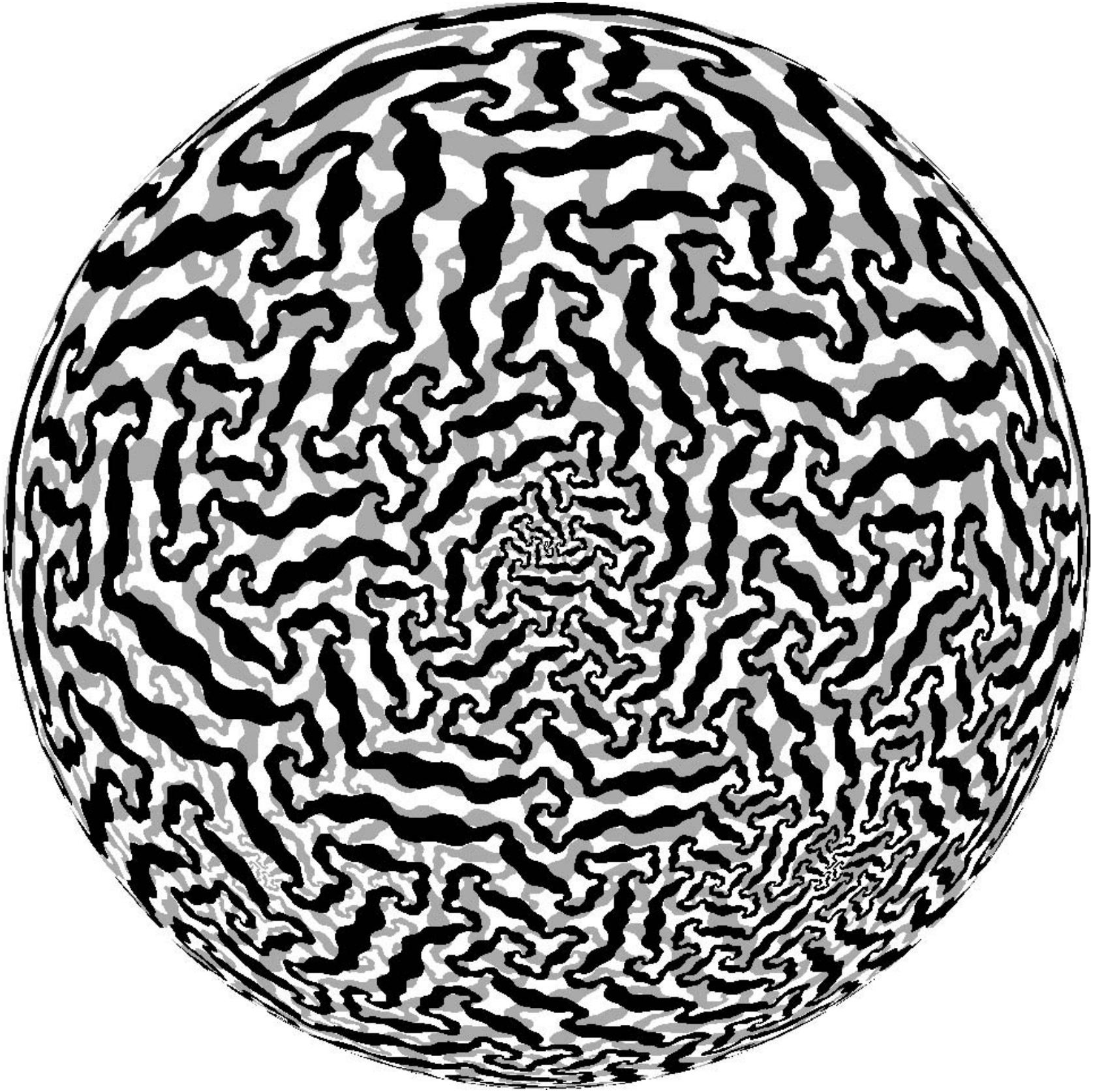}}}
\newcommand{\drawfiglattesfour}{\scalebox{.25}{\includegraphics{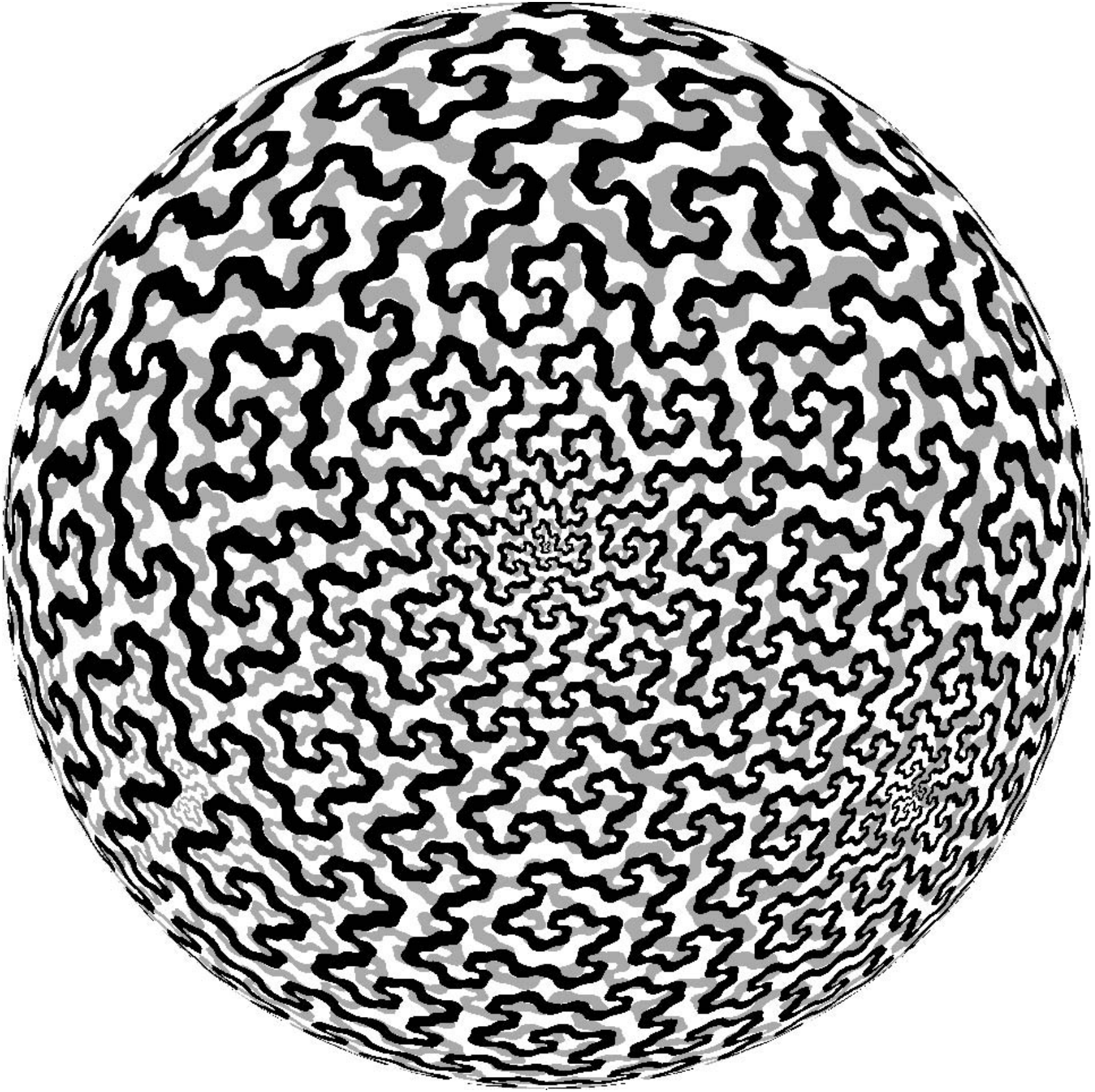}}}

\newcommand{\drawfigbabyrabmatebabyrab}{\scalebox{.25}{\includegraphics{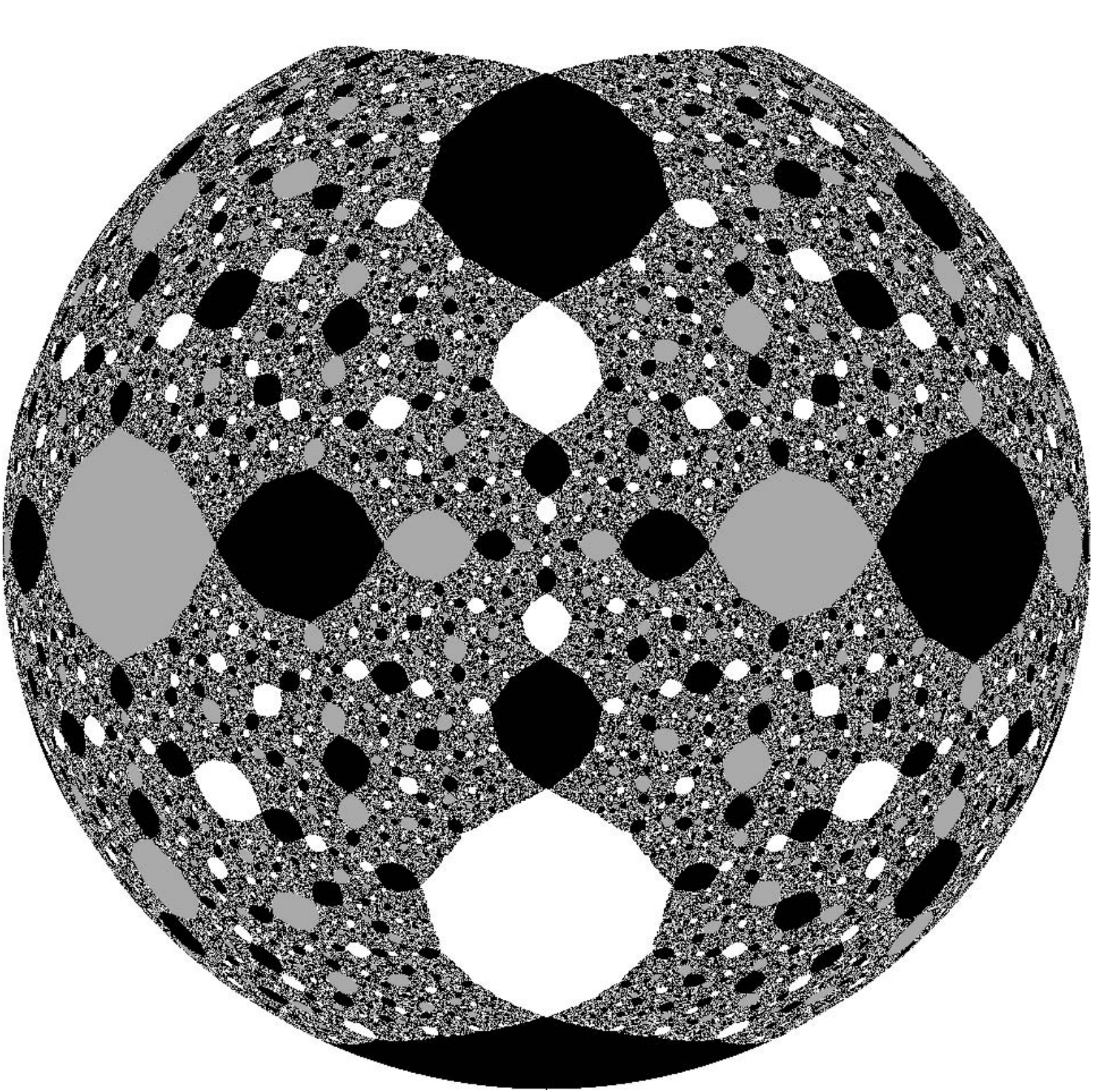}}}
\newcommand{\drawfigbabyrabMline}{\scalebox{.25}{\includegraphics{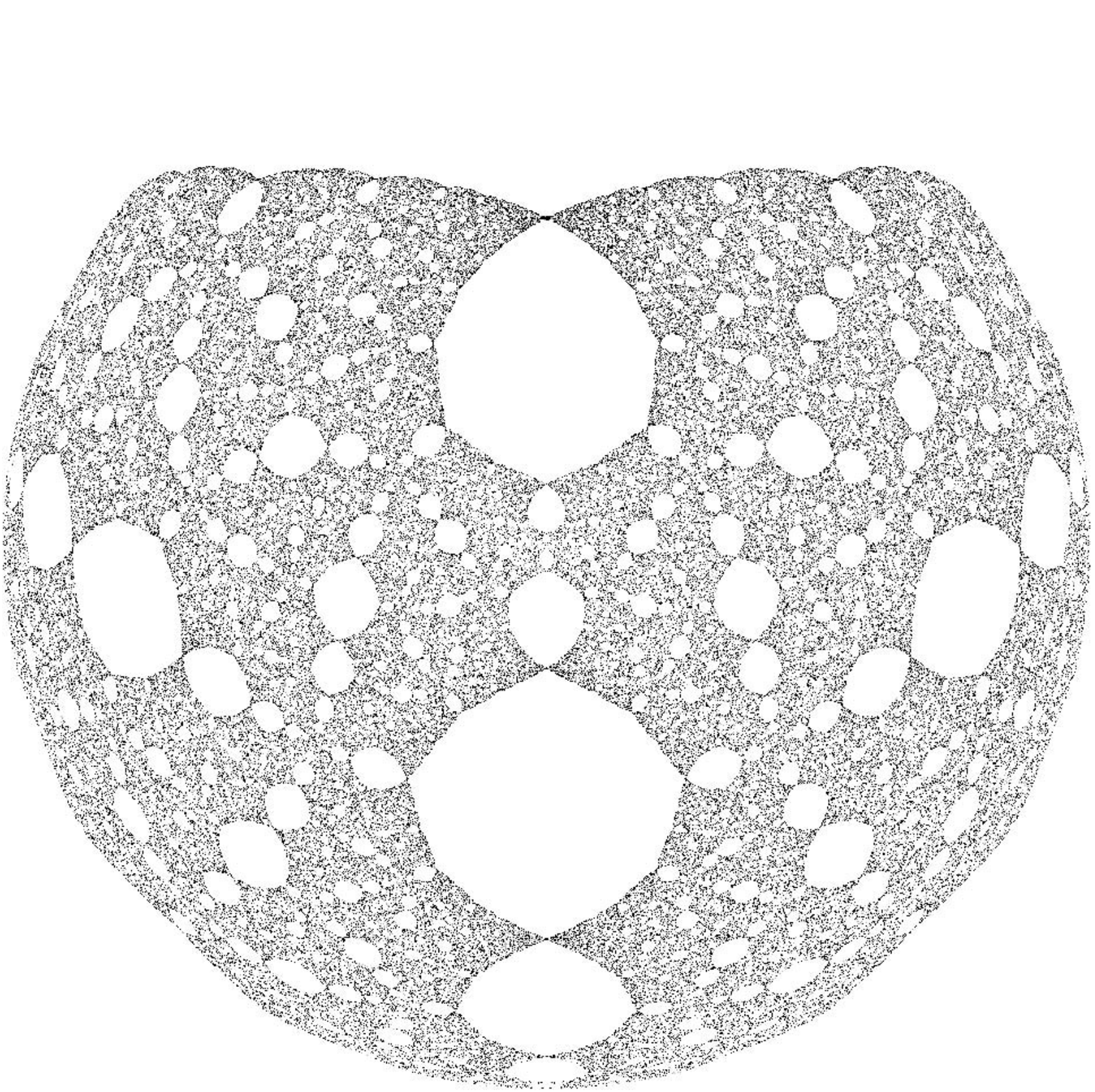}}}
\newcommand{\drawfigbabyrabMlineapprox}{\scalebox{.25}{\includegraphics{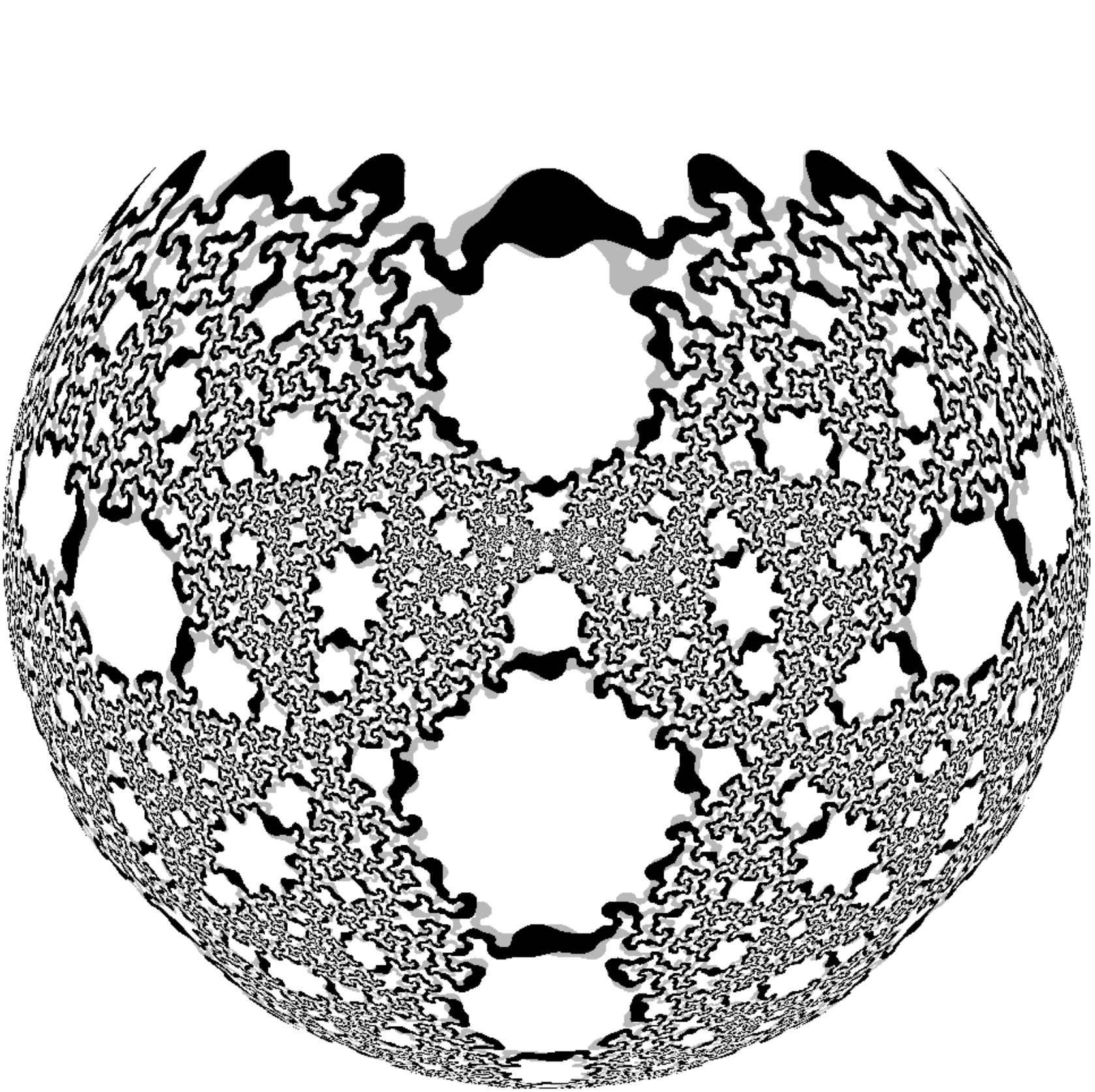}}}

\newcommand{\drawfigseqone}{\scalebox{.2}{\includegraphics{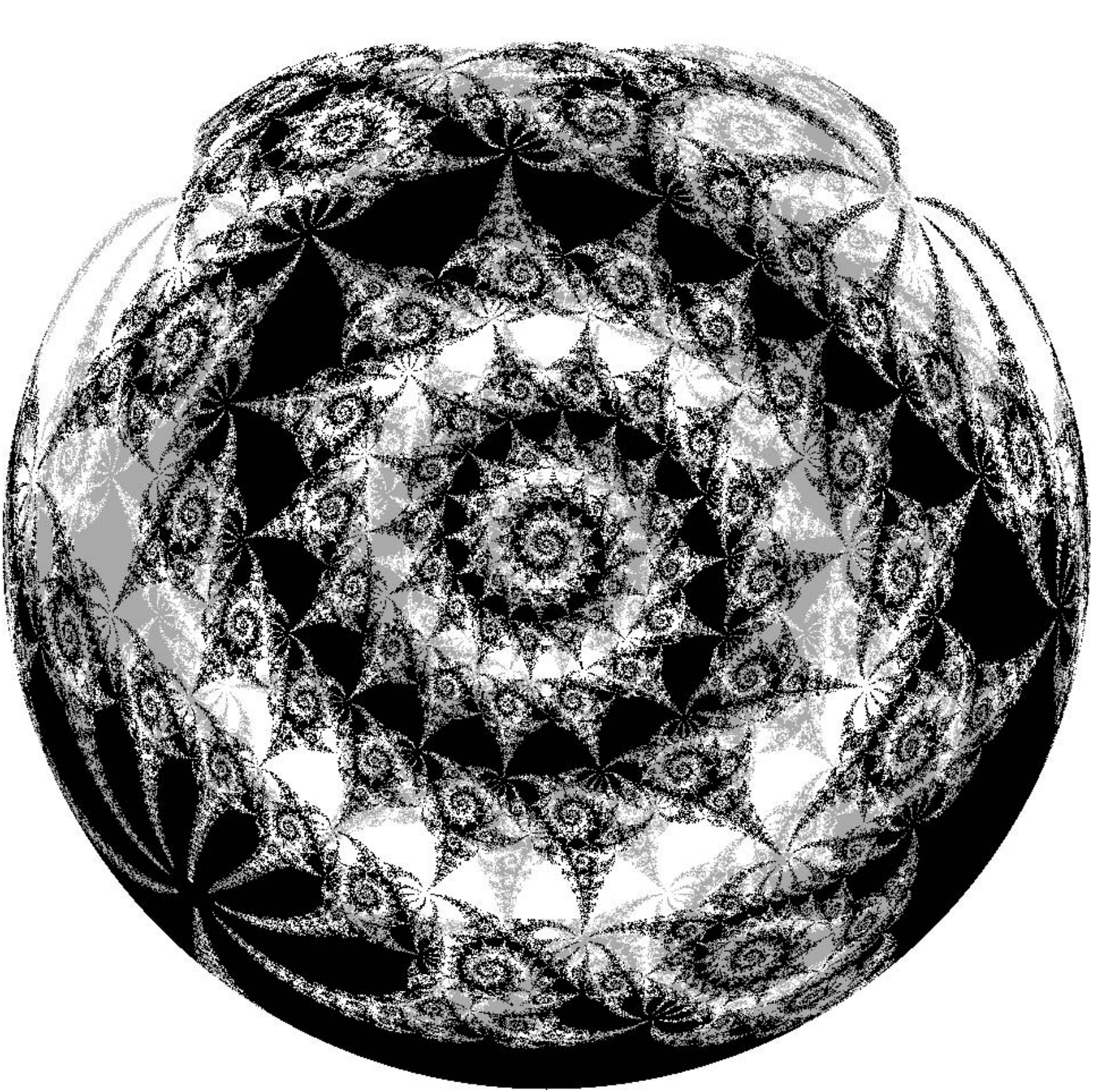}}}

\newcommand{\drawfigseqtwo}{\scalebox{.2}{\includegraphics{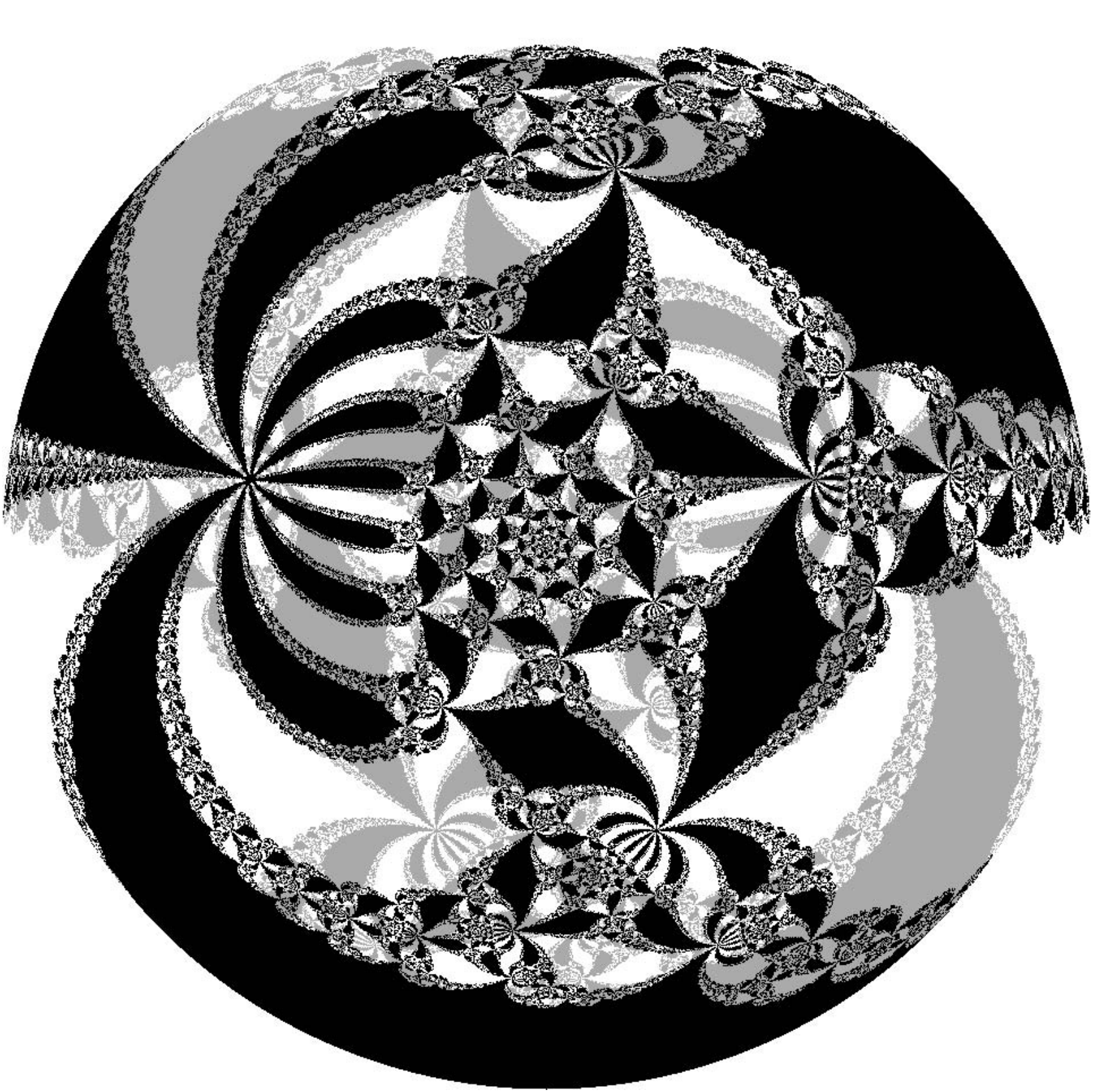}}}

\newcommand{\drawfigseqthree}{\scalebox{.2}{\includegraphics{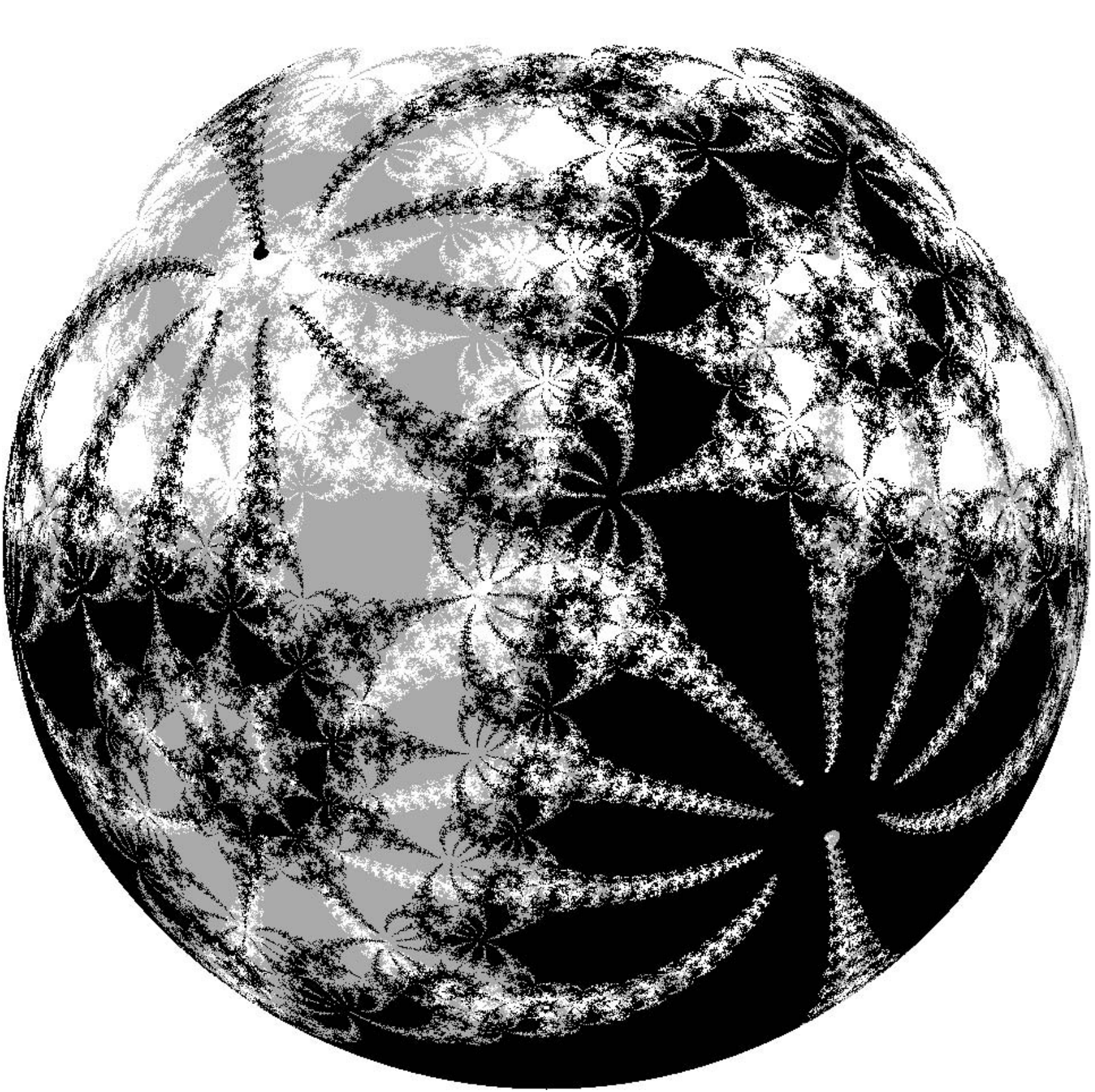}}}

\newcommand{\showcomments}{yes}

\newsavebox{\commentbox}
\newenvironment{comment}%
{\ifthenelse{\equal{\showcomments}{yes}}%
{\footnotemark
    \begin{lrbox}{\commentbox}
    \begin{minipage}[t]{1.25in}\raggedright\sffamily\small
    \footnotemark[\arabic{footnote}]}
{\begin{lrbox}{\commentbox}}}%
{\ifthenelse{\equal{\showcomments}{yes}}%
{\end{minipage}\end{lrbox}\marginpar{\usebox{\commentbox}}}
{\end{lrbox}}}

\newcommand{\hide}[1]{} % Contents of hide is eaten

\begin{document}

\title{The Medusa algorithm \\for polynomial matings}

\author[S.H. ~Boyd]{Suzanne Hruska Boyd
%$^2$
}
\address{Department of Mathematical Sciences\\
University of Wisconsin Milwaukee\\
PO Box 413\\
Milwaukee, WI 53201\\
USA}
\email{shruska@msm.umr.edu}

\author[C. ~Henriksen]{Christian Henriksen
%$^2$
}
\address{Department of Mathematics, Build. 303\\
Technical University of Denmark\\
DK -- 2800 Kgs. Lyngby}
\email{christian.henriksen@mat.dtu.dk}

\date{\today}

\begin{abstract}
The {\em Medusa algorithm} takes as input two postcritically finite quadratic polynomials and outputs  the quadratic rational map which is the mating of the two polynomials (if it exists).  Specifically, the output is a sequence of approximations  for the parameters of the rational map, as well as an image of its Julia set. Whether these approximations converge is answered using Thurston's topological characterization of rational maps.   

This algorithm was designed by John Hamal Hubbard, and implemented in 1998 by Christian Henriksen and REU students David Farris, and Kuon Ju Liu. In this paper we describe the algorithm and its implementation, discuss some output from the program (including many pictures) and related questions. Specifically, we include images and a discussion for some shared matings, Latt\`{e}s examples, and tuning sequences of matings.
\end{abstract}

\maketitle

\markboth{\textsc{S. Boyd and C. Henriksen}}
  {\textit{Medusa and matings}}

%\footnotetext[1]{$^{,2}$Research supported in part by a grant from the National Science Foundation.}

%%%%%%%%%%%%%%%%%%%%%%%%%%%%%%%%%%%%%%%%%%%%%%%%
\section{Introduction}
\label{sec:introduction}
%%%%%%%%%%%%%%%%%%%%%%%%%%%%%%%%%%%%%%%%%%%%%%%%

The study of the dynamics of rational maps of the Rieman sphere is greatly facilitated by the fact that a wide variety of dynamical phenomena can be illustrated using only the quadratic family $P_c(z) = z^2 +c$.  Of course most general theorems about rational maps have examples in the quadratic family, but further, in some cases the dynamics of a quadratic polynomial appear within a rational map. The most basic example of this phenomena is through polynomial-like behavior. In addition, there are several ways to combine two (or more) quadratic polynomials to produce rational maps whose dynamics can be described via a combination of the quadratic polynomial dynamics. Probably the first such example was a {\em polynomial mating} discovered by Adrien Douady \cite{Douady1983}.

In order to define matings, first we
%in \S\ref{sec:backgroundmatings}.
 must step back to quadratic polynomials.
It is simple to write a computer program which, given a $c$, will compute (approximately) the orbit of any given point under the quadratic polynomial $P_c$. To illustrate the overall behavior one draws the {\em filled Julia set}, $K_c$, the set of points whose orbit under $P_c$ does not tend to $\infty$.  This also illustrates the {\em Julia set}, $J_c$, the topological boundary of $K$.  (See \S\ref{sec:background}, Figure~\ref{fig:juliarays} for a sample $J_c$.) We may examine experimentally the dynamics of one map at a time with such a program.

The next natural step is to understand how the dynamics changes with a change in the parameter, $c$.
We organize the parameter space by defining ${M}$, the {\em Mandelbrot set}, as the set of all $c$ in $\CC$ for which the Julia set $J_c$ is connected (see \S\ref{sec:background}, Figure~\ref{fig:mandelbrot}).  By Fatou's fundamental dichotomy theorem, this is equivalent to the set of all $c$ such that the orbit of the critical point $0$ under $P_c$ lies in $K_c$.   
Thus it is also a simple matter to generate a picture of ${M}$, and a program which will draw the Julia set $J_c$ when a parameter $c$ in ${M}$ is selected.  After a brief investigation with such a program, one sees intriguing patterns, and a relationship between ${M}$ and the Julia sets of its children, the quadratic polynomials.

In addition to the definition of ${M}$, many basic results in the theory of the iteration of rational functions support the premise that the behavior of the critical orbit is crucial for describing the dynamics.  
 The dynamics are most amenable to  analysis when the polynomial $P_c$ is {\em postcritically finite (PCF)}, i.e., the orbit of the critical point $0$ is finite.  A key technique in giving a mathematical description of the patterns of quadratic polynomials turns out to be combinatorics.  For a postcritically finite quadratic polynomial, we can build a labelled graph, called a {\em spider}, which gives  a combinatorial description of the dynamics of the polynomial.  This is described in \S\ref{sec:backgroundrays}.

The reverse problem, of starting with a combinatorial spider and producing a quadratic polynomial $P_c$ (i.e., producing a parameter $c$) whose dynamics are given by that model, is solved by the {\em spider algorithm}. The spider algorithm is an iterative procedure, based on Thurston's topological characterization of rational maps \cite{DH_T}, and is described fully in \cite{HS-spider}.  

The main subject of this paper is the {\em Medusa algorithm}, which takes two combinatorial spiders, glues them together in a certain manner (hence the name Medusa), then runs a sort of double spider algorithm which, if it converges, produces 
%the parameters defining 
a rational map which is the {\em mating} of the two quadratic polynomials associated with the originally inputted spiders, see Theorem~\ref{thm:main}.

%The {\em Medusa algorithm} takes as input two PCF quadratic polynomials and outputs  the quadratic rational map which is the mating of the two polynomials (if it exists).  Specifically, the output is a sequence of approximations  for the parameters of the rational map, as well as an image of its Julia set. Whether these approximations converge is answered using Thurston's topological characterization of rational maps. In this paper we describe the algorithm and its implementation, and discuss some output from the program and related questions.

John Hamal Hubbard designed the Medusa algorithm, based on Thurston's theory (\cite{DH_T}) and the foundational theory of polynomial matings developed by Douady, Hubbard, Shishikura, Rees, Tan Lei and others (\cite{Douady1983, Rees1992, TanLei1992, Shish2000}, see \S\ref{sec:backgroundmatings}).  
The computer program implementing the algorithm was written under Hubbard's direction by David Farris, Christian Henriksen and Kuon Ju Liu, in a 1998 \textit{summer research experience for undergraduates program.} The full source code for Medusa is available for download at \cite{CUweb}.

Some progress has been made in the study of  polynomial matings since 1998, however there are still many intriguing questions.  The goals of experimental software like Medusa are to help form conjectural answers to existing questions, as well as inspire new questions.  After explaining the algorithm and implementation, in the final section of this paper we provide several examples of images we created using Medusa, which serve to illustrate and examine several  of the phenomena of matings.  
Specifically, we include images and a discussion for some Latt\`{e}s examples, shared matings, and tuning sequences of matings.
We hope this paper will energize future researchers to study polynomial matings, and we expect Medusa is of service in advancing the field.

\medskip\noindent{\bf Organization of sections.} 
In \S\ref{sec:background} we provide needed prerequisite material on the dynamics of quadratic polynomials and polynomial matings.
In \S\ref{sec:Medusa}, we describe the Medusa algorithm and its implementation, proving Theorem~\ref{thm:main}.  
The final section, \S\ref{sec:examples}, contains examples of output from the program related to a few areas of interest in the study of matings.

\subsection*{Acknowledgements}  The authors thank Dierk Schleicher, Adam Epstien and Tan Lei for inspiring discussions and advice on how to write this paper. All images of Julia sets of quadratic polynomials were generated with the Otis fractal program \cite{Otis}.

%%%%%%%%%%%%%%%%%%%%%%%%%%%%%%%%%%%%%%%%%%%%%%%%
\section{Background}
\label{sec:background}
%%%%%%%%%%%%%%%%%%%%%%%%%%%%%%%%%%%%%%%%%%%%%%%%

%========================
\subsection{Notation}
\label{sec:notation}
%=========================
We write $\Ch = \CC \cup \{\infty\}$ for the Riemann sphere, i.e., the one point compactification of the complex plane, endowed with the complex structure with respect to which the identity restricted to $\CC$ is a chart, and $z\mapsto 1/z$ a conformal isomorphism. We write $\Stwo$ for $\Ch$ viewed as a topological manifold, i.e., not equipped with a canonical complex structure.

%========================
\subsection{Quadratic polynomials and combinatorics}
\label{sec:backgroundrays}
%=========================

%The {\em Julia set} $J$ of a polynomial map of $\CC$ is the topological boundary of $K$, the set of points whose orbit under $f$ does not tend to $\infty$.  

If $K_c$ is connected, then there is a unique conformal isomorphism 
$$
\psi_c \colon \Ch - \overline{\DD} \to \Ch - K_c,
$$ such that $\psi_c'(\infty) = 1.$  This map conjugates $w \mapsto w^d$ to $P_c$.
The curve $\mathcal{R}_t (c)= \mathcal{R}_t= \{ \psi(r e^{2\pi it}) \colon r > 1\}$ is the {\em external ray of angle $t$}.  
For a postcritically finite polynomial the filled in Julia set $K_c$ is locally connected, and then $\psi_c$ extends continuously to the boundary.
If we parameterize the circle by $\RR / \ZZ$, then the map $\psi_c$ on the boundary becomes 
$$\gamma_c \colon \RR / \ZZ \rightarrow J_c,$$
and $\gamma_c$ is a semiconjugacy of multiplication by two to $f$, i.e.,  $\gamma_c(2t) = P_c(\gamma(t))$.
Then $\gamma_c(t)$ is called the {\em landing point} of $\mathcal{R}_t(c)$.
Call $\gamma_c$ the {\em Carath\'{e}odory map of $P_c$}. See Figure~\ref{fig:juliarays} for a picture of a Julia set and some external rays.

%---------------------------------------------------
\begin{figure}
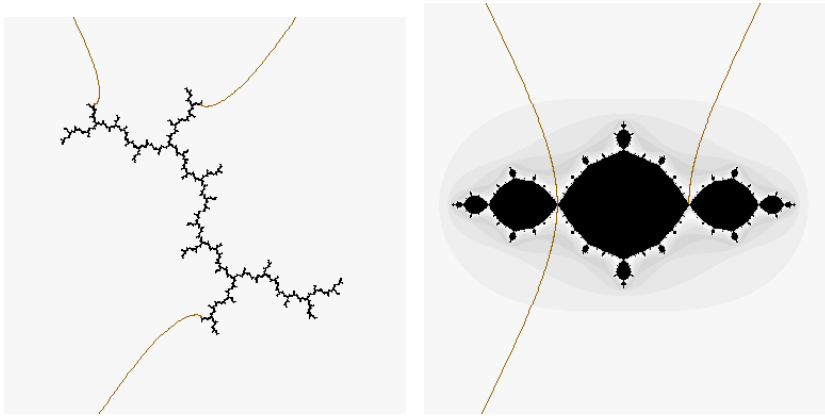

\begin{center}
 \drawfigJonesixth \ \ 
   \drawfigJonethirdonesixth
\caption{\label{fig:juliarays}
Left: the Julia set of $z \mapsto z^2+i$, and critical orbit rays $1/6, 1/3, 2/3$.
Right: the Julia set of $z \mapsto z^2-1$, and critical orbit rays $1/3, 2/3$, plus $1/6$ for comparison.
}
\end{center} 
\end{figure}
%---------------------------------------------------

Given a postcritically finite quadratic polynomial, $P_c$, choose $\theta_c \in \RR / \ZZ$ so that $\mathcal{R}_{\theta_c}$ is the external ray associated with the critical value, $c$.  That is, $\mathcal{R}_{\theta_c}$ lands at $c$, if $c \in J_c.$ Otherwise the critical point is periodic. If the critical point is fixed, take $\theta_c = 0.$ If the critical point is periodic of period $n>1,$ the critical value is contained in the immediate basin $U$ of a superattracting cycle and there exists a pair of rays landing at the root of $U$ whose closure seperates the critical value from the other points in the critical orbit. Take $\theta_c$ to be one of the two angles corresponding to this pair of rays.

We can use $\gamma_c$ and $\theta_c$ to create a simple combinatorial model of the critical orbit.

Given a rational number $\theta \in \RR / \ZZ$, following Hubbard and Schleicher (\cite{HS-spider}) we define the {\em standard $\theta$-spider, $\Sbb_{\theta} \subset \Ch$} by:
$$
\Sbb_{\theta} = \{ r e^{2 \pi i 2^{j-1} \theta}  \colon r \geq 1, j = 1, 2, \ldots \} \cup \{ \infty \}.
$$
 See the image on the left in Figure~\ref{fig:spider} for an example, it shows the spider for one of the Julia sets of Figure~\ref{fig:juliarays}.
One may view this as a spider, with legs the rays emanating from the unit circle which are in the orbit of $\theta$ under angle doubling, and body the point at infinity.  

Since $\gamma_c$ semi-conjugates $P_c$ to angle doubling, $\gamma_c$ maps $\Sbb_{\theta_c}$ to the union of $\mathcal{R}_{\theta_c}$ and its images under $P_c$, plus the point at infinity.  Note if $\theta$ is rational, then it has finite orbit under angle doubling, so the spider has a finite number of legs. Similarly, if $P_c$ is postcritically finite, then $\theta_c$ will be rational.
We denote the endpoints on the unit circle of the spider legs by $z_j = e^{2 i \pi 2^{j-1} \theta}$.

%---------------------------------------------------

\begin{figure}
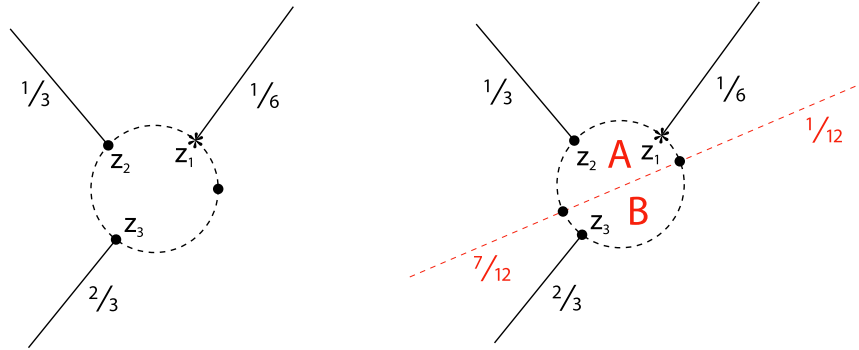

\begin{center}
   \drawfigspiders
\caption{\label{fig:spider}
Left: the spider for $\theta = 1/6$.  The critical orbit is $(1/6 \mapsto 1/3 \mapsto 2/3 \mapsto 1/3)$.  Right: the kneading sequence for this spider is $K(1/6) = A \ \overline{AB}$.  This spider models $f(z) = z^2+i$, whose Julia set is shown in Figure~\ref{fig:juliarays}.}
\end{center}
\end{figure}

%---------------------------------------------------
The spider illustrates the critical orbit.  Using this diagram we can also create a sequence called the {\em kneading sequence of $\theta$} which records information about the order of the critical orbit in this diagram.  Take the plane containing the spider $\Sbb_{\theta}$, and cut along the line composed by the rays of angle $\theta/2$ and $(\theta+1)/2$.  Label by $A$ the open half of the plane containing $\theta$, label the other open half $B$.  See the right hand image of Figure~\ref{fig:spider}.
Label the ray of angle $\theta$ by $*_a$, and the ray of angle $(\theta+1)/2$ by $*_b$.  For any angle $t$, its $\theta$-itinerary is the infinite sequence of labels from $(A, B, *_a, *_b)$ corresponding to the position in the labelled plane of the points in the forward orbit of $t$ under angle-doubling.  The {\em kneading sequence of $\theta$}, denoted $k(\theta)$, is the $\theta$-itinerary of the angle $\theta$.  Note a symbol $*_n$ appears in this sequence if and only if $\theta$ is periodic under angle doubling.  

%We will also use the more general space $\mathcal{S}_{\theta}$ of $\theta$-spiders, defined (as in \cite{HS-spider}) to be the quotient of the space \begin{eqnarray*} \mathcal{S}^0_{\theta} = \{  \varphi \colon \Sbb_{\theta} \rightarrow \Ch  & \ | \ &  \varphi(\infty) = \infty, \varphi(z_1) = 0, \varphi \text{ is injective, continuous, } \\ & & \text{ and respects the circular order at } \infty \}, \end{eqnarray*} by the equivalence relation generated by two equivalence relations:  both scaling and continuous deformation of the legs with endpoints fixed.

In this paper, we are interested in combining and comparing quadratic polynomials.  In order to keep track of the dynamics of various maps we are studying, we use the discovery of
Douady and Hubbard (see \cite{DH1}) on how $\theta_c$ relates to the position of $c$ in the Mandelbrot set.
They show the Mandelbrot set, ${M}$, is connected, with simply connected complement in $\Ch$, hence
there is a unique conformal isomorphism
$\Psi_M \colon \Ch - {M} \rightarrow \Ch - \overline{\DD}$ which fixes $\infty$ and such that $\Psi_M'(\infty)=1$.   Then $\Psi_M$ defines external rays outside of ${M}$, by images of straight rays outside of the disk.  It happens that for any rational angle $\theta = p/q$, the map $\Phi_M$ extends radially to the boundary, to define a landing point $c(\theta)$ for the ray of angle $\theta$. Given a postcritically finite polynomial $P_c$ to which we associate the angle $\theta_c,$ then the parameter ray of angle $\theta_c$ will either land at $c$ (in the preperiodic case) or at the root of the hyperbolic component of $M$ that has $c$ as a center (in the periodic case).
%   The miracle is that $\theta$ is the angle of the external ray associated with the critical value $c$ of $f_{c(\theta)}$, i.e., in the notation above, given $P_c$ a postcritically finite quadratic polynomial, we have $c = c(\theta_c)$.
%
For example, for the basilica, $f (z) = z^2 - 1$, the external rays associated with the critical value $-1$ is of angle $1/3$ and $2/3$.  The parameter rays of angle $1/3$ and $2/3$ lands on the Mandelbrot set at the root point of the bulb containing the basillica (the real bulb). Figure~\ref{fig:mandelbrot} shows the Mandelbrot set and some external rays.

%---------------------------------------------------

\begin{figure}
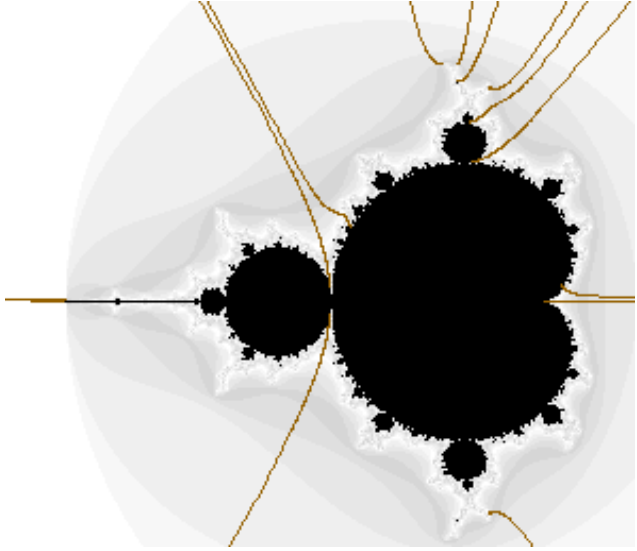

\begin{center}
   \drawfigmandel
\caption{\label{fig:mandelbrot}
The Mandelbrot set,  i.e., the set of all $c$ in $\CC$ for which the Julia set $J_c$ is connected, shown above in black, together with the external rays: $0, 1/511, 1/7, 10/63, 1/6, 3/14, 1/5, 1/4, 169/511, 1/3, 255/511, 1/2, 2/3, 5/6.$
}
\end{center}
\end{figure}

%---------------------------------------------------

%=======================================
\subsection{Mating quadratic polynomials}
\label{sec:backgroundmatings}
%=======================================

Let $f_n(z) = z^2 + c_n, n=1,2$ be two quadratic polynomials, with Julia sets $J_n$. Assume each $J_n$ is locally connected, and  $\gamma_n$ is the Carath\'{e}odory map of $f_n$. 
Define $
\mathcal{K}
 = K_1 \sqcup K_2 / \backsim$ to be the quotient space of the disjoint union of $K_1$ and $K_2$ in which for each $t \in \RR / \ZZ$, we identify $\gamma_1(t)$ with $\gamma_2(-t)$.  In other words, we obtain a topological space $\mathcal{K}$ by gluing $K_1$ and $K_2$ together along their boundaries via $\gamma_1(t) \backsim \gamma_2(-t)$.   Consider this definition while viewing Figure~\ref{fig:juliarays}. In general one might imagine $\mathcal{K}$ as some bizarre balloon animal (possibly with infinitely many body segments), but we will see below that in many cases, $\mathcal{K}$ is simply a sphere.  
 
On the space $\mathcal{K}$, define the map $f_1 \mate f_2$ by $f_n$ on $K_n, n=1,2$.  Since $\gamma_n$ semiconjugates $f$ to multiplication by two on $J_n$, this map is well-defined and continuous (no matter how bizarre the space $\mathcal{K}$ may be).

If there is a quadratic rational map $F$ which is topologically conjugate on $\Ch$ to $f_1  \mate f_2$ on $\mathcal{K}$, then $F$ is called a {\em mating} of $f_1$ and $f_2$. We denote this relationship by $F \cong f_1 \mate f_2$, and in this case say the mating of $f_1$ and $f_2$ {\em exists}.  The conjugacy $h \colon \mathcal{K} \to \Ch$ is required to be an orientation preserving homeomorphism which is holomorphic on the interiors of each $K_n$.  It is believed that if $F$ exists, it is unique up to M\"{o}bius conjugation.

Note that a mating of any quadratic polynomial $f_1$ with $f_2(z) = z^2$ yields $F \cong f_1$.

Results of Rees, Shishikura, and Tan Lei (\cite{Rees1992, TanLei1992, Shish2000}) show that 
whether the mating of two PCF quadratic polynomials $f_1$ and $f_2$ exists 
can be answered in terms of the location of $c_1$ and $c_2$ in parameter space.  
The fundamental existence theorem is:

\begin{thm} 
\label{thm:mateexist}
If $f_1, f_2$ are PCF quadratic polynomials, TFAE:
\begin{itemize}
\item $\mathcal{K}$ is homeomorphic to the sphere $S^2$;
\item there exists a quadratic rational map $F$ which is the mating of $f_1$ and $f_2$;
\item $c_1$ and $c_2$ do not belong to complex conjugate limbs of the Mandelbrot set, ${M}$.
\end{itemize}
\end{thm}

We refer the reader to Milnor's book \cite{Mil-book} for detailed background on the dynamics of polynomial maps of $\CC$, and his article \cite{Mil-pasting} for a more complete discussion of the  definition of mating and its subtleties, a discussion of many foundational results on matings, and a detailed analysis of an interesting example of mating.

%%%%%%%%%%%%%%%%%%%%%%%%%%%%%%%%%%%%%%%%%%%%%%%%
\section{From Thurston's Algorithm to the Medusa Algorithm}
\label{sec:Medusa}
%%%%%%%%%%%%%%%%%%%%%%%%%%%%%%%%%%%%%%%%%%%%%%%%

Thurston's algorithm is a proof that given a branched covering $g$ of the sphere there exists a rational map $F$ that is Thurston equivalent to $g$ unless there exists a Thurston obstruction. The proof can be made into an iterative procedure computing a sequence of complex structures and rational maps $F_n$ which, when properly normalized, converges to $F$. In this section we see that we can take $g$ to be a model of the mating of two quadratic rational maps, and extract finite dimensional but crucial information about the complex structures produced by Thurston's Algorithm so that the sequence $F_n$ can be recovered. This is the heart of the Medusa Algorithm. Because of the finite dimensional information needed to run the algorithm, it lends itself to actual computation.

\subsection{The Theory}

\subsection*{Normalizing matings}
Assume $f_1, f_2$ are postcritically finite quadratic polynomials and $F \cong f_1  \mate f_2$.
Each $f_n$ has one critical point $0$, which lies in $K_n$.  Thus $F$ has two distinct critical points. By conjugating $F$ with a mobius transformation we can arrange that the critical point coming from $f_1$ is at the origin, the other critical point at infinity and the two glued-together beta fixed points are at $1.$
Therefore we know that any such mating belongs to the following family of maps.

\begin{notation} We normalize the rational maps which are matings by:
\begin{equation} \label{eqn:Family}
\Fam = \{F \text{ rational of degree 2}~|~0, \infty \text{ are critical points and } F(1)=1\}.
\end{equation}
\end{notation}

Note that every rational map of degree two is conjugate to (at least one) member of $\Fam.$

The following innocent lemma, which is trivial to prove, is of fundamental importance to why there is such a thing as the Medusa Algorithm.

\begin{lem} \label{FamilyLemma}
Given two distinct points $u, v \in \Ch \setminus \{ 1 \}$ there exists a unique $F \in \Fam$ so that $F(0) = u$ and $F(\infty) = v.$
\end{lem}

The lemma shows that there is some magic to quadratic rational maps. Normalized in the way described, we just need the position of the two critical values (and which correspond to which critical point) to uniquely determine the map. We don't need any extra combinatorial information.

\begin{proof}
We prove the lemma in the case where $u, v$ are different from infinity.
The case where either $u$ or $v$ equals infinity is just as easy and left to the reader.
First notice that
$F: z \mapsto  \frac{(u-1)v z^2 - u(v-1)}{(u - 1)z^2-(v-1)} \in \Fam,$
has the desired properties, so we need to show that this is the only such map in $\Fam.$ 
Since the origin and infinity are critical points, we can write
$$F(z) = \frac{a z^2 + c}{b z^2 + d}.$$
That $1$ is fixed, $F(\infty) = v$ and $F(0)=u$ implies that $a - vb = 0,$
$c - ud = 0$ and $a-b+c-d=0.$ When either $u$ or $v$ is different from $1$ the matrix
$$
\left[
\begin{array}{cccc}
1 & -v & 0 & 0 \\
0 &  0 & 1 & -u \\
1 & -1 & 1 & -1 
\end{array}
\right]
$$
has rank $3.$ It follows that every solution to the three equations
can be written $(a,b,c,d) = \lambda ((u-1)v, u-1, -u(v-1), -(v-1))$ for some
$\lambda \in \CC,$ and therefore $F$ is uniquely determined. 
\end{proof}

In the following we will write $F_{u, v}$ for the map given by the lemma.

\subsection*{The Standard Medusa}
We now build a model for the mating $F = f_1 \mate f_2$ of the two postcritically finite quadratic maps $f_1, f_2.$ We start by defining the standard Medusa.

\begin{defn}  \label{defn:stdmed}
Let $\theta_1, \theta_2 \in \ZZ$ be the two rational numbers we associate to $f_1$ and $f_2,$ as in \S\ref{sec:backgroundrays}. Define the $(\theta_1, \theta_2)$ \textit{standard Medusa}
$\MM(\theta_1, \theta_2) \subset \Stwo$ to be the union of the unit circle $\Sone,$ the \textit{interior legs}
$$\{\rho \exp(2 i \pi  2^j \theta_1)~|~\frac{1}{2} \leq \rho \leq 1, j=1, 2, \ldots \}$$
and the \textit{exterior legs}
$$\{\rho \exp(-2i \pi 2^j\theta_2)~|~1 \leq \rho \leq 2, j = 1, 2, \ldots \}.$$
\end{defn}

Defined in this way we have that $z \mapsto 1/z$ maps
$\MM(\theta_2, \theta_1)$ bijectively to $\MM(\theta_1, \theta_2).$
The endpoints of the interior legs we denote by $x_j$, and the endpoints of the exterior legs we denote by $y_j$, hence
$$x_j = 2 \exp(2i\pi 2^j \theta_1), j=1, 2, \ldots, \text{ and  }$$
$$y_j = 1/2 \exp(-2i\pi 2^j \theta_2), j=1, 2, \ldots .$$

We can think of the standard Medusa as a coupling of  two standard spiders $\Sbb_{\theta_1}, \Sbb_{\theta_2}$, where the bodies have been cut away, then the two are glued along the cut.
See Figure~\ref{fig:Medusaform} for a schematic diagram of this process.

%---------------------------------------------------

\begin{figure}
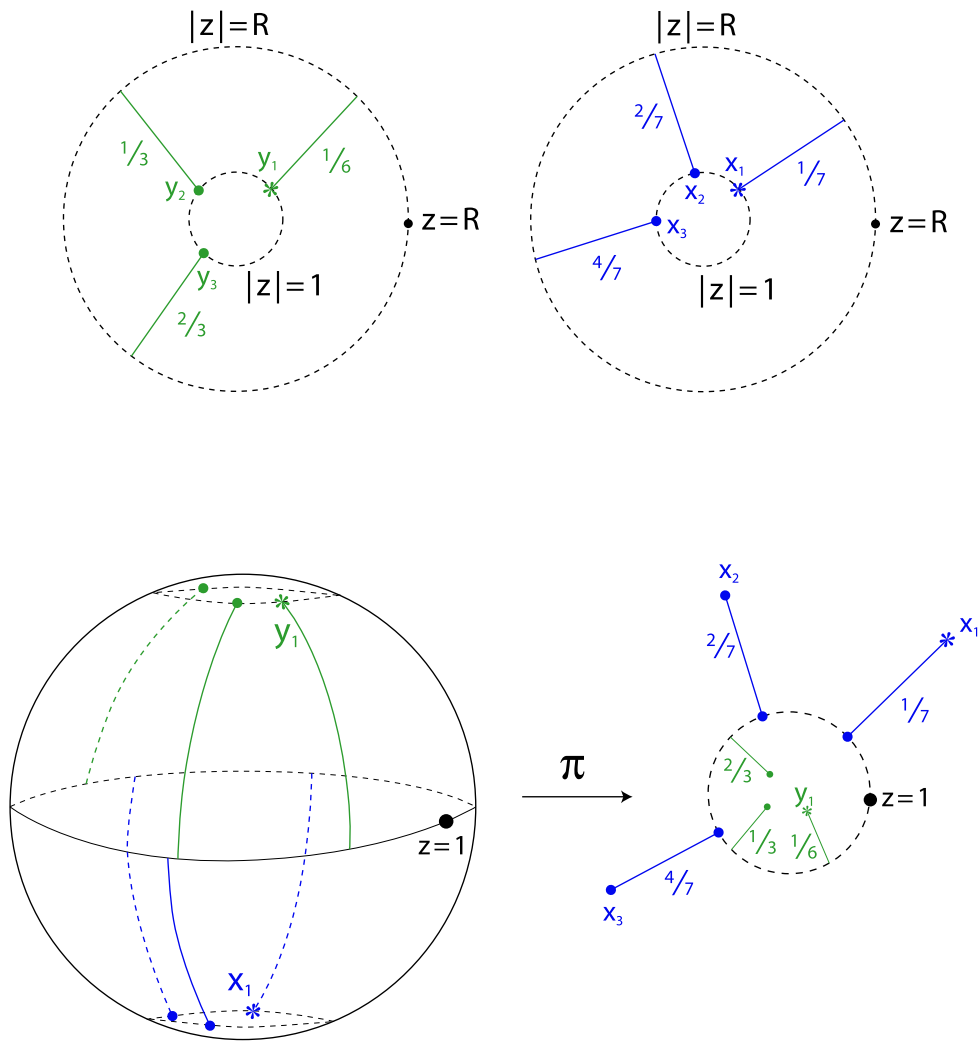

\begin{center}
  \drawfigformmedusa
     \caption{
 \label{fig:Medusaform} Above is a schematic of the process of mating $1/6$ with $1/7$.  The upper figures are the truncated spiders, the lower left is the Medusa on the sphere, and the lower right is the projection of the Medusa to the plane.
}
\end{center}
\end{figure}

%--------------------------------------------------

\subsection*{Thurston Matings}
Recall that two postcritically finite branched coverings $F: \Stwo \to \Stwo$ and $g : \Stwo \to \Stwo$ with postcritical sets $P_F$ and $P_g$ are called \textit{Thurston equivalent} if there exists orientation preserving homeomorphisms $\phi$ and $\psi$ such that $\phi$ restricted to $P_F$ maps bijectively onto $P_g$ and $\psi^{-1} \circ \phi$ is isotopic to the identity on $\Stwo$ rel.\ $P_F.$

We proceed to define a branched covering $g$ of $\Stwo$ by itself that in nondegenerate cases is Thurston equivalent to the mating $F = f_1 \mate f_2$. Let
$g|_{\MM(\theta_1, \theta_2)}$ be the angle doubling map
$r \exp(i\phi) \mapsto r \exp(2i \phi).$ Extend $g$ smoothly to a degree two branched covering of the sphere so that:
\begin{enumerate}
\item $g : \DD \to \DD$ is a degree two branched coveing with critical value at $x_1,$ and
\item $g : \Stwo \setminus \overline{\DD} \to \Stwo \setminus \overline{\DD}$ is a degree two branched covering with the critical value at $y_1.$ 
\end{enumerate}
Denote by $\omega_1$ the critical point of $g$ in $\DD$ and by $\omega_2$ the critical point
of $g$ in $\Stwo \setminus \overline{\DD}.$ Notice that $\omega_i$ coincides with an endpoint of a leg if and only if $\theta_i$ is periodic under angle doubling,
$\theta \mapsto 2\theta \mod 1.$

Notice that if we redefine $g$ outside the unit circle to by setting it equal to $z \mapsto z^2$ here, we obtain a map that is Thurston equivalent to $f_1.$ Similarly, if we instead redefine $g$ inside the unit circle so it restricts to $z\mapsto z^2$ here, we obtain a mapping that is Thurston equivalent to $f_2.$ Hence it is reasonable to view $g$ as our branched covering model of the mating $F.$  Shishikura \cite{Shish2000} guarantees convergence in the nondegenerate case:

\begin{defn}  \label{defn:strongmate}
Let $f_1, f_2$ be PCF quadratic polynomials not in complex conjugate limbs of $M$.  If the two critical orbits of $F \cong f_1 \mate f_2$ are disjoint, then $f_1$ and $f_2$ are called {\em strongly mateable}.
\end{defn}

\begin{thm} [\cite{Shish2000}]  \label{thm:strongmate}
If $f_1, f_2$ are strongly mateable, then $g$ is Thurston equivalent to the mating $F \cong f_1 \mate f_2$.
\end{thm}

Thurston's algorithm is an iterative process that will give us a sequence of rational maps converging to $F$ when $F$ and $g$ are Thurston equivalent. Using $g$ as our model map, it works as follows. 
\begin{figure}
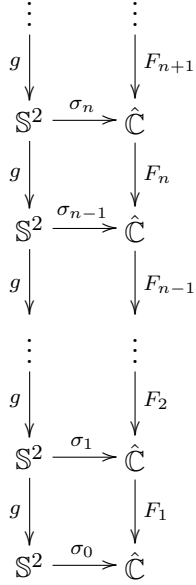

$$
\diagram
\vdots \dto_g & \vdots \dto^{F_{n+1}} \\
\Stwo \rto^{\sigma_n} \dto_g & \Ch \dto^{F_n} \\
\Stwo \rto^{\sigma_{n-1}} \dto_g & \Ch \dto^{F_{n-1}} \\
\vdots \dto_g & \vdots \dto^{F_2}  \\
\Stwo \rto^{\sigma_1} \dto_g & \Ch \dto^{F_1} \\
\Stwo \rto^{\sigma_0} & \Ch 
\enddiagram
$$
\caption{\label{fig:comdiag}A commutative diagram representing the maps involved in Thurston's algorithm.}
\end{figure}
Let $\sigma_0 : \Stwo \to \Ch$ be an orientation preserving homeomorphism mapping $\omega_1$ to $0,$ $\omega_2$ to $\infty$ and fixing $1.$ Recursively define $\sigma_n$ and $F_n$ as follows for $n = 1, 2, \ldots$ Interpret $\sigma_{n-1}$ as a global chart defining a complex structure on $\Stwo.$ This complex structure can be pulled back by $g.$ Indeed, since $g$ is a local homeomorphism everywhere except at $\omega_i,$ $i=1,2$ we can just compose restrictions of $g$ with $\sigma_{n-1}.$ The complex structure defined in this way can be uniquely extended to the missing points $\omega_1, \omega_2.$ By the uniformization theorem $\Stwo$ equipped with the pullback complex structure is conformally equivalent to $\Ch.$ So let $\sigma_n : \Stwo \to \Ch$ be the conformal isomorphisms 
 and normalize it so $\omega_1$ is mapped to $0,$ $\omega_2$ to $\infty$ and $1$ is fixed. By construction $F_n$ defined by the composition  $\sigma_n \circ g \circ \sigma_n^{-1}$ is holomorphic. The sequence of maps constructed can be illustrated by the commutative diagram shown in Figure~\ref{fig:comdiag}.

In principle Thurston's algorithm solves our problem, the sequence of generated maps rational maps should converge to our mating. However the set of possible complex structures on $\Stwo$ is beyond actual computations, so we need to adapt the algorithm to allow for this. This is exactly what Hubbard's Medusa Algorithm does for us.

%-----------------------------------------------------

\subsection*{The Medusa Algorithm}

Notice that each map $F_n$ in Thurston's algorithm (in the strongly mateable described) is a degree two rational map fixing $1$ and having the origin and infinity as critical points. In other words, $F_n \in \Fam.$ By Lemma \ref{FamilyLemma} we just need to know to where $0$ and $\infty$ are mapped to identify $F_n$. Hence we don't need all the information contained in the sequence of complex structures to find $F_n,$ it is enough knowing $\sigma_{n-1}$ restricted to the standard Medusa $\MM(\theta_1, \theta_2).$ Motivated by this we make the following definition.

\begin{defn} \label{defn:medsp} Set
$$\Mspace_0(\theta_1, \theta_2) = \{ \sigma |_{\MM(\theta_1, \theta_2)}
~|~ \sigma \in \textrm{homeo}^+(\Stwo \to \Ch) \text{ and normalized}\},$$ 
where normalized here means that $\sigma(\omega_1) = 0, \sigma(\omega_2) = \infty, \sigma(1) = 1,$ and define the \textit{Medusa space}
$\Mspace$  as the quotient of $\Mspace_0$ with the equivalence relation that identifies $\sigma_1$ and $\sigma_2$ if and only if the two maps are isotopic rel $\{ x_1, x_2, \ldots\} \cup \{y_1, y_2, \ldots, \}$ through mappings in $\Mspace_0.$
\end{defn}

Notice there is a natural projection $\pi$ from the complex structures on $\Stwo$ onto $\Mspace.$ Given a complex structure $\Sigma$ we know by the uniformization theorem that there exists a conformal isomorphism $\sigma : (\Stwo, \Sigma) \to \Ch$ which we can normalize so that $\omega_1$ maps to $0,$ $\omega_2$ to infinity and $1$ is fixed. We let
$\pi(\Sigma)$ equal the equivalence class of $\sigma|_{\MM(\theta_1, \theta_2)}$ in $\Mspace(\theta_1, \theta_2).$

%We will also use the more general space $\mathcal{M}(\theta_1, \theta_2)$ of $(\theta_1, \theta_2)$-Medusas, defined to be the quotient of the space
%\begin{eqnarray*}
%\mathcal{M}_0(\theta_1, \theta_2) = \{ \varphi \colon \MM(\theta_1, \theta_2) \rightarrow \Ch 
%& \ | \ &
%\varphi(y_1) = 0 (?), 
%\varphi \text{ is injective, continuous,  } \\ & & \text{ and satisfies what?? } \},
%\end{eqnarray*}
%by the equivalence relation generated by the two different equivalence relations of scaling and %continuously deforming the legs and crown, keeping the endpoints fixed.  \begin{comment} right? %\end{comment}

One can show that there is a natural bijection between $\Mspace(\theta_1, \theta_2)$ and the Teichm\"uller space of $\Stwo \setminus \{x_1, x_2, \ldots, y_1, y_2, \ldots, 1 \},$ so Medusa space is a finite dimensional complex manifold in a natural way.

Mappings in Medusa space can be lifted. More precisely we have the following lemma.

\begin{lem} \label{MedusaIterate}
Let $s_{n-1} \in \Mspace_0(\theta_1, \theta_2)$ be given. Set $u_n = s_{n-1}(x_1), v_n = s_{n-1}(y_1)$ and let $F_{u_n, v_n} \in \Fam$
be the unique mapping as in Lemma~\ref{FamilyLemma}. Then there is a unique mapping $s_n \subset \Mspace_0(\theta_1, \theta_2)$ such that the following diagram commutes.
$$
\diagram
\MM(\theta_1, \theta_2) \rto^{s_n} \dto_g & \Ch \dto^{F_{u_n, v_n}} \\
\MM(\theta_1, \theta_2) \rto^{s_{n-1}} & \Ch
\enddiagram
$$
If $s_{n-1}'$ and $s_{n-1}$ represent the same element in $\Mspace(\theta_1, \theta_2)$ then
the lifts $s_n'$ and $s_n$ also represent the same element in $\Mspace(\theta_1, \theta_2).$
\end{lem}

\begin{proof} Since the simple closed curve $\gamma' = \sigma_{n-1}(\Sone)$ seperates one critical point $0$ and its image $u_n = F_{u_n, v_n}(0)$ from the other critical point $\infty$ and its image $v,$ the preimage $\gamma$ of $\gamma'$ by $F_{u_n, v_n}$ is a simple closed curve and
$F_{u_n, v_n} : \gamma \to \gamma'$ is a two to one covering map. Identify the fundamental group on $\Sone$ with $\ZZ$ so that a curve having index $1$ with respect to $0$ correspond to $+1 \subset \ZZ.$ Do similarly for $\gamma$ and $\gamma'.$ Then the induced map $g_* : \ZZ \to \ZZ$ is multiplication by two. Since $s_{n-1}$ extends to a homeomorphism that maps $\omega_1$ to $0$
$(s_{n-1})_* : \ZZ \to \ZZ$ is the identity. Finally, $F_{u_n, v_n}$ maps the bounded component of $\Ch \setminus \gamma$ onto the bounded component of $\Ch \setminus \gamma'$ which implies that $(F_{u_n, v_n})_* : \ZZ \to \ZZ$ is multiplication by $+2.$
Hence $(s_{n-1} \circ g)_* : \pi_1(\Sone) \to \pi_1(\gamma')$ has the same image as
$(F_{u_n, v_n)})_* : \pi_1(\gamma) \to \pi_1(\gamma').$ It follows by a fundamental theorem of algebraic topology that there exists a covering map $s_n : \Sone \to \gamma$ so that $g \circ s_{n-1} = s_{n-1} \circ F_{u_n, v_n}$ on $\Sone,$ and this lift is unique when we require that $s_n(1) = 1.$ We can extend $s_n$ to $\MM(\theta_1, \theta_2)$ by lifting each leg seperately, in the way that agree with how $s_n$ is defined on the circle. In this way we have obtained a homeomorphism $s_n$ mapping $\MM(\theta_1, \theta_2)$ to its image, and we must show that $s_n \in \Mspace_0(\theta_1, \theta_2).$ 
However, since $F_{u_n, v_n}$ maps the bounded (unbounded) part of $\Ch \setminus \gamma$ to the bounded (unbounded) $\Ch \setminus \gamma'$ the image of an interior (exterior) leg is interior (exterior), so we can extend $s_n$ to a orientation preserving homeomorphism of the sphere as required.

We still need to show uniqueness of $s_n.$ For $s_n$ to be an element of $\Mspace_0(\theta_1, \theta_2)$ we must have $s_n(1) = 1$ and that uniquely determines $s_n$ on $\Sone.$ Knowing $s_n$ on the unit circle means we know to where the base point of the legs must lift and therefore there is only on extension to $\MM(\theta_1, \theta_2)$ such that
$g \circ s_{n-1} = s_{n-1} \circ F_{u_n, v_n}.$

Finally suppose that $s_{n-1}'$ and $s_{n-1}$ represent the same element in $\Mspace(\theta_1, \theta_2)$ and let $s_n, s_n' \in \Mspace_0(\theta_1, \theta_2)$ be the two unique lifts. By assumption there exists an istopy connecting $s_{n-1}'$ to $s_{n-1},$ through maps in $\Mspace_0(\theta_1, \theta_2).$ This isotopy can be lifted to an isotopy connecting $s_n$ and $s_n'.$ Each map in the isotopy maps $1$ to $1$ so as before we can prove that it is an element of $\Mspace_0(\theta_1, \theta_2).$ 
\end{proof}

Let a starting point $S_0 \in \Mspace(\theta_1, \theta_2),$ be given. The Medusa algorithm consists of repeatedly applying Lemma \ref{MedusaIterate} to get a sequence $S_n \in \Mspace(\theta_1, \theta_2)$ and rational maps $F_{u_n, v_n} \in \Fam$ for $n = 1, 2, \ldots.$
The beauty of the algorithm is that we produce the same sequence of rational maps that Thurston's algorithm produces.

\begin{thm} \label{EquivalenceTheorem}
Suppose $\pi(\sigma_0) = S_0.$ Then the rational maps produced by Thurston's algorithm equal those produced by the Medusa algorithm, $F_n = F_{u_n, v_n}.$ Futhermore $\pi(\sigma_n) = S_n,$ $n=1, 2, \ldots$
\end{thm}

\begin{proof} Assume $\pi(\sigma_{n-1}) = S_{n-1},$ this is case when $n=1$ by assumption. Then
$s_{n-1} = \sigma_{n-1} |_{\MM(\theta_1, \theta_2)}$ is a representative of $S_{n-1}.$
Now $F_n \in \Fam$ maps $0$ to $\sigma_{n-1}(x_1)$ and $\infty$ to $\sigma_{n-1}(y_1).$
So too does $F_{u_n, v_n}.$ Hence, by Lemma \ref{FamilyLemma} $F_n = F_{u_n, v_n}.$
We have that $\sigma_n |_{\MM(\theta_1, \theta_2)} \in \Mspace_0(\theta_1, \theta_2),$ is a lift
of $s_{n-1}.$ So by the uniqueness part of Lemma \ref{MedusaIterate} $\pi(\sigma_n) = S_n.$
The theorem now follows by induction.
\end{proof}
 
%+++++++++++++++++++++++++++++++++++++++++++++++++++++++++

Now we can justify the Medusa algorithm, by combining Theorems~\ref{thm:mateexist}, \ref{thm:strongmate}, and~\ref{EquivalenceTheorem}:

\begin{thm} \label{thm:main}
If $f_1$ and $f_2$ are strongly mateable, then the Medusa algorithm converges to the mating $F  \cong f_1 \mate f_2$.
\end{thm}

In practice, the algorithm seems to converge without assuming the maps are strongly mateable.
Thus we expect that a stronger theorem holds; namely, it should be the case that anytime $f_1$ and $f_2$ are PCF quadratic polynomials  in complex conjugate limbs of $M$, the Medusa algorithm should converge to the mating.   The case not covered by Thurston's theorem is when two polynomials that are not in complex conjugate limbs have a mating with only one critical orbit.  In this case naively running the Medusa algorithm produces a sequence of Medusas which does not converge (rather tends to the boundary of the Teichmuller space), but the obstruction points (the critical orbits becoming identified) are all pushed together upon iteration of the algorithm, hence the sequence of rational maps seems to converge to the mating.  To prove this stronger result one could  investigate how the maps in the Medusa algorithm are converging as the boundary of the Medusa space is approached.  We expect the techniques of Nikita Selinger's PhD thesis \cite{Nik} on convergence at the boundary of Teichmuller space could be adapted to solve this question, and leave this future result to the interested reader. 

%---------------------------------------------------
\subsection{The Implementation}
%---------------------------------------------------
The point of the Medusa algorithm is that it lends itself to implementation as a computer program. The implementation is an adoption of the implementation of the spider algorithm to the more general setting of quadratic rational maps.

To initiate the program, the user inputs two rational angles $\theta_1, \theta_2$.  The implementation defines an initial Medusa $s_0 \colon \MM (\theta_1, \theta_2) \to \Ch$, say close to the identity.  

To describe our matings,  we define a chart on $\Fam$ by letting $R_{a,b} : z \mapsto \frac{az^2+(1-a)}{b z^2 + (1-b)},$ $(a,b) \in \CC^2 \setminus \{ (z,z)~|~z\in \CC \}.$ In this way we parametrize all the maps in $\Fam.$ Supposing that $F \in \Fam$ maps $0$ to $u$ and $\infty$ to $v,$ we let $a = \frac{v(u-1)}{u-v}$ and $b = \frac{u-1}{u-v}.$ Then $R_{a,b} = F = F_{u, v}.$

We represent a mapping $s : \MM(\theta_1, \theta_2) \to \CC$ by several lists of points in $\Ch.$ One list represent the image of the unit circle, and the other lists represent the images of the legs. Also we always let the list of points representing the image of the unit circle start with the point $1.$

We adopt the convention that two consecutive points in the image of the unit circle or in a leg is connected by an arc of circle. For the points on the image of the circle or on the interior legs the circle chosen is that through $s(y_1),$ and the arc of circle chosen is the one connecting the two points and omitting $s(y_1).$ For consecutive points on the exterior legs adopt the convention that they are connected by the arc of the circle through the points and $s(x_1).$ The arc is the one that connects the two points and omits $s(x_1).$

Clearly, with the information contained in the lists of points and the convention just mentioned we can reconstruct, not $s,$ but the isotopy class of $s.$

An iteration consists of finding the class of the pullback of $s_{n-1}$ (as in Lemma \ref{MedusaIterate}). As in the implementation of the spider algorithm we break the process down into three steps: a pullback step, a rectifying step and a pruning step.

\subsection*{Pullback.}
Given $s_{n-1}$ as lists of points as described we first find $F_{u_n, v_n} = R_{a_n, b_n}.$ This corresponds to solving 
\begin{equation} \label{eqn:medusainitialparams}
\frac{1-a}{1-b} = u_n = s_{n-1}(x_1) \  \text{ and } \ \frac{a}{b} = v_n = s_{n-1}(y_1).
\end{equation}
In other words 
$$a_n = \frac{(u_n - 1)v}{u-v} \ \text{ and } \ b_n = \frac{u_n-1}{u-v}.$$
 Notice that $R_{a_n, b_n}$ is the composition of a Mobius transformation with $z \mapsto z^2.$ Hence, pulling back a point consists of first pulling it back by a Mobius transformation $M_n$ and then by the square. The question that needs to resolved is, what branch of the squareroot do we need to choose.

First we pullback the points corresponding to the image of the unit circle. Suppose that we have pulled back a point $z_k$ and obtained the point $w_k$ and want to pullback the next point in the list $z_{k+1}.$ Pulling back first by the Mobius transformation we get that the circle through $z_k, z_{k+1}$ and $v_n$ becomes a circle through $M_n^{-1}(z_k), M_n^{-1}(z_{k+1})$ and $\infty$ i.e. a line. Since the arc of circle connecting the two points was chosen to be the one that did not contain $v_n$ the pullback of the arc of circle by the Mobius transformation becomes simply a line segment between $M_n^{-1}(z_k), M_n^{-1}(z_{k+1}).$ The preimage of a line by the square is a hyperbola, the two branches of which are contained in opposite quarter planes. Hence knowing one preimage $w_k,$ we need to choose the square root so that $w_k$ and $w_k+1$ lies in the same halfplane.

So the pullback the points corresponding to the circle we construct to lists, $A, B.$ The first element of $A$ is $1$ and the first element of $B$ is $-1,$ i.e. the two preimages of $1$ by $R_{a_n, b_n}.$ This was the first step. Next we iterate through the remaing points in the list. The $k$'th step consists in finding the two preimages of $z_k,$ call them $w_k$ and $w_k'.$ If the last inserted point in the list $A$ lies in the same quarter plane as $w_k$ then we insert $w_k$ in $A$, and $w_k'$ in $B$. Otherwise we insert $w_k$ in $B$ and $w_k'$ in the $A$. It is easy to verify that the points in the list $A$ are images of the points on the unit circle of angles in the interval $0 \leq \theta < \pi,$ whereas the points in $B$ correspond to angles $\theta$ with $\pi \leq \theta < 2\pi.$ Having pulled back all the points we can concanate the two list so we get one list (starting with the point $1$) representing the image of the circle by $s_n$. Notice that this list contains 
 twice the points of the one we have just pulled back.

Next we pull back the interior legs. The leg corresponding to angle $\theta$ is the preimage of the leg corresponding to angle $2\theta.$ If $0 \leq \theta < \pi$ the point in the list $A$ list that is the preimage of the anchor point of the leg of angle $2\theta$ will be the anchor point of the new $\theta$ leg, otherwise it will be the corresponding point in the list $B$. Hence we have already computed (and can localize) the pull-back of the first point in the leg. Hence, as before we can pull back the rest of the leg, we need to chose the square root so the consecutive points lies in the same halfplane.

Pulling back the outer legs is essentially the same, except that now pulling back by $M_n$ two consecutive defines a arc of circle, where the circle goes through $0.$ However since $z \mapsto z^2$ commutes with $z \mapsto 1/z$ we can write the squre as the composition of $1/z,$ $z^2$ and then $1/z$ again. Hence pulling back by $M_n$ and then making the change of coordinates
$w = 1/z$ we are back in the same situation as the one we were facing when pulling back the interior legs.

In this way we obtain a list of point representing the map $s_n.$ However, the points are now connected by arcs of hyperbolas and not arcs of circles. The next step, rectifying, remedy this situation.

\subsection*{Rectifying.}
Perhaps a better word for the second part of an iteration would be \textit{circlyfying.} We want to bring us back to the starting position where consecutive points in the lists are connected by arcs of circles. This is the most delicate part of the implementation. What we want to do is replace the arcs of hyperbolas with arcs of appropiate circles without changes the isotopy class of the corresponding element in $\Mspace(\theta_1, \theta_2).$
So given two consecutive points $z_1, z_2$ we want to see if there is a homotopy from an arc of hyperbola to an arc of circle so that the intermediate curves does not cross any of the \textit{distinguished points} $s_n(x_1), s_n(x_2), \ldots, s_n(y_1), s_n(y_2), \ldots\}.$ It is rather tedious so we will only outline how it is done. The circle and the hyperbola are two (real) quadratic curves and we first find their intersection. This can be boiled down to finding the roots of a degree $4$ equation in one real variable. However, since we know that $z_1$ and $z_2$ lies on both curves, we can do a division of polynomial and the remaing points (if any) can be found by solving a quadratic equation. The most difficult case when the branch of hyperbola containg $z_1$ and $z_2$ intersect the circle in four points. Then the union of the circle and the branch of hyperbola cuts the plane into six parts. By elementary geometric reasoning, one can find exactly to which of the six parts a given point belongs, and this knowledge is enough to decide if the homotopy exists.

If the homotopy exists then we can move on, but if it doesn't we need to do something. What we do is to subdivide the arc of hyperbola in two halves, $z_1, \zeta$ and $\zeta, z_2$ and recursively rectify each half. In case we are not dealing with a leg terminating at a distinguished point, then by compactness the distinguished points are a definite distance away from the arc of hyperbola between $z_1, z_2.$ Given any $\epsilon > 0,$ any fine enough subdivision of the arc of hyperbola, $z_1, \zeta_1, \zeta_2, \ldots, \zeta_k, z_2$ will satisfy, that if we replace the parts of hyperbolas with arcs of circles we will stay with a spherical $\epsilon$ neighborhood of the original arc of hyperbola. Hence, we are able to rectify after adding only a finite number of points. In the case that the arc of hyperbola terminates in a distinguished point $z_2$ then we are dealing with the image of a leg. It is not difficult to see that we do not change the isotopy class of $s_n$ by allowing the homotopy to cross $z_2.$ In practice, this means that when rectifying a leg, we do not consider the endpoint of the leg a distinguished point, and we are sure that we can rectify adding only a finite number of points.

\subsection*{Pruning.}
After pulling back and rectifying, we have new lists of point representing $s_n,$ but the number of points representing the image of the unit circle has at least doubled. This means that unless we do something we will run out of memory in a finite number of iterations.

What we do is \textit{pruning} which amounts to checking each point $z_2$ that is not the attachment point or terminal point of the leg whether it can be removed without changing the isotopy class of the represented map. In practice this means checking whether two arcs of circles, one  through $z_1$ and $z_2$ the other through $z_2$ and $z_3,$ can be replaced by an arc of circle going from $z_1$ to $z_3$ without changing isotopy class. Using a Mobius transformation to change coordinates the question becomes whether a line segment $(w_1, w_2)$ and a line segment $(w_2, w_3)$ can be homotopied to a line segment $(w_1, w_3)$ without crossing distinguished points, a question that can be easily answered.

%---------------------------------------------------

\begin{figure}
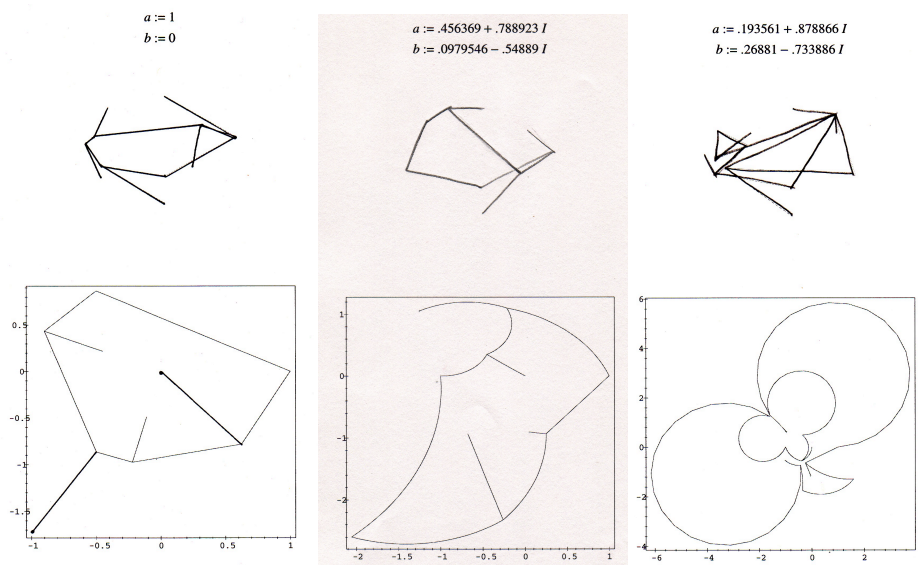

\begin{center}
   \drawfigrealmedusazero \drawfigrealmedusatwo \drawfigrealmedusatwenty 
   \caption{
\label{fig:actualmedusas}
Each of the three columns above shows Maple output of the actual Medusas used in the iteration of the Medusa algorithm for the mating of $1/7$ with $1/3$ (rabbit mate basilica). In each column, the top figure is the Medusa on the sphere, the lower figure is the Medusa projected onto the plane.  Leftmost is the initial Medusa, central is after 2 steps, rightmost is after 20 steps. 
% \textit{These are scanned in printouts from 1998.  If we get the program running again, we may want to replace these with screen captures.}
}
\end{center}
\end{figure}

%----------------------------------
\subsection*{Drawing the Julia set.}
In addition to producing a sequence of maps $R_{a_n, b_n}$ converging to the mating, the Medusa algorithm can be used to draw successive approximations to the Julia set of the mating.  
At the beginning of the program, a ``painted'' sphere $\mathcal{K}_0$ is created, with each point in the upper hemisphere painted black, and each point in the lower hemisphere painted white (or clear).
At each iteration of the algorithm, given parameters $a_m, b_m$ 
and a painted sphere $\mathcal{K}_{m-1}$ (i.e., a sphere with each point marked one of black or white),
the program computes the pull back of $\mathcal{K}_{m-1}$ by $R^{-1}_{a_m,b_m}$, to create $\mathcal{K}_{m}$.

When the sequences  ${(a_m, b_m)}$ converge, then $R_{a_m, b_m}$ converges to $R_{a,b} \cong f_1 \mate f_2$, and $\mathcal{K}_m$ converges to $\mathcal{K}$, with white or clear marking the Julia set of $f_1$, and black the Julia set of $f_2$.  

For example, let $c_{1/4}$ be the parameter which is the landing point in the Mandelbrot set of the external ray of angle $1/4$ ($c_{1/4} \approx -0.228 + 1.115 i$).  This is a tip point on the rabbit bulb.  The mating of $z^2 + c_{1/4}$ with itself exists, and is studied in detail in \cite{Mil-pasting}.  
In this case the Julia set of the mating is the entire sphere, so the approximations $\mathcal{K}_n$ drawn by Medusa are particularly interesting.  Figure~\ref{fig:lattessequence} shows approximations $\mathcal{K}_{6}$, $\mathcal{K}_{10}$, and $\mathcal{K}_{14}$ for this mating. Also see \S\ref{sec:lattes} for other similar examples. 

%---------------------------------------------------

\begin{figure}
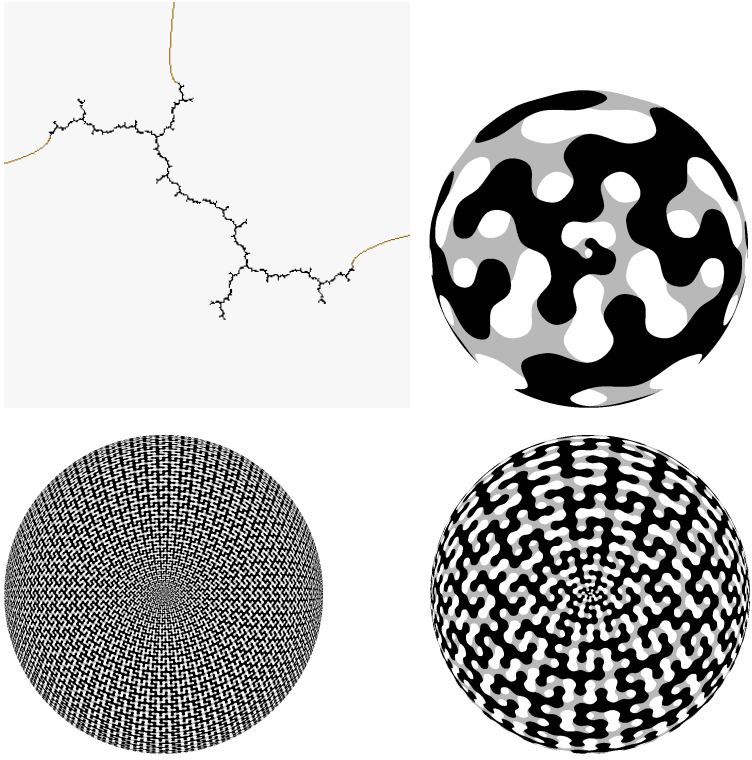

\begin{center}
 \drawfigJonefourth   \ \ \drawfiglatone  \\
 \vspace{.125in}
 \drawfiglatthree \ \  \ \ \ \ \ \ \ \ \ \drawfiglattwo
   \caption{
 \label{fig:lattessequence}
Upper Left: the Julia set of $f (z) = z^2 + c_{1/4}$, where $c_{1/4}$ is the landing point in the Mandelbrot set of the external ray of angle $1/4$ (so $c_{1/4} \approx -0.228 + 1.115i$), shown with critical orbit rays $1/4, 1/2, 0$.  Clockwise around (upper right, lower right, lower left):  approximations on the sphere $\mathcal{K}_6, \mathcal{K}_{10}, \mathcal{K}_{14}$, respectively, to the Julia set of $f (z) = z^2 + c_{1/4}$ mated with itself.
%
%From left to right are shown successive approximations on the sphere to $\mathcal{K}_6, \mathcal{K}_{10}, \mathcal{K}_{14}$, to the Julia set of $f (z) = z^2 + c_{1/4}$ mated with itself, where $c_{1/4}$ is the landing point in the Mandelbrot set of the external ray of angle $1/4$ (so $c_{1/4} \approx -0.228 + 1.115i$).   
}
\end{center}
\end{figure}

%---------------------------------------------------

The full source code for Medusa is available for download at \cite{CUweb}.
There are still a few bugs, most notably: when mating with a $p/q$
where $q$ is even, the algorithm will converge properly for a few steps,
then start diverging. 

\section{Examples}
\label{sec:examples}
%%%%%%%%%%%%%%%%%%%%%%%%%%%%%%%%%%%%%%%%%%%%%%%%

In this section we discuss several types of matings with different properties.  
For simplicity, we will refer to a PCF quadratic polynomial simply by its rational angle $\theta_c = p/q$, or sometimes $f_{p/q}$. 

%===================================
\subsection{Simple examples}
%===================================
We explain our first example of an image of a mating produced by the Medusa algorithm in detail.   We will mate the two quadratic polynomials shown in Figure~\ref{fig:rabmatebas}: $f_1$ will be the rabbit, $1/7$, and $f_2$ will be the basilica, $1/3$.  

%---------------------------------------------------
\begin{figure}
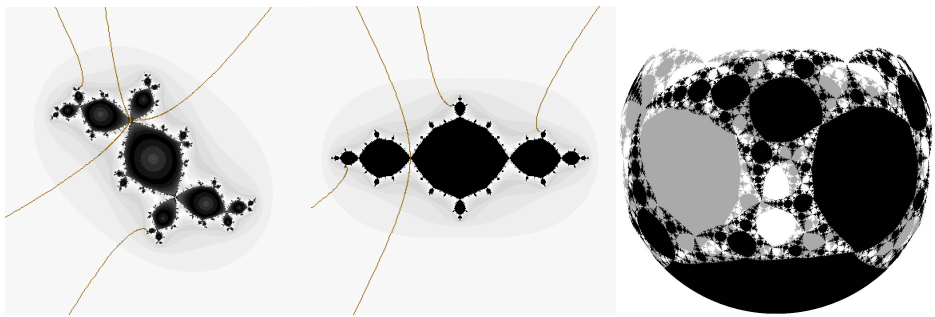

\begin{center}
   \drawfigJoneseventhonethird \drawfigJonethirdoneseventh 
  \drawfigrabmatebas
   \caption{
 \label{fig:rabmatebas}
From left to right:  The Julia set of the rabbit, critical angle $1/7$, then the Julia set of the basilica, critical angle $1/3$, both shown with both sets of critical orbit rays ($1/7, 2/7, 4/7, 1/3, 2/3$) for comparison; finally, the mating $1/7$ mate $1/3$ on the sphere, with $1/3$ in black, and $1/7$ clear.
}
\end{center}
\end{figure}
%---------------------------------------------------

Let $F = f_1 \mate f_2 = 1/7 \mate 1/3$.  The rightmost sphere in Figure~\ref{fig:rabmatebas} illustrates the Julia set of the mating $F$.    Due to our normalization (Equation~\ref{eqn:Fnormalform}), the critical point $0$ of $f_1$ is always at $z=0$ in the sphere, shown as the south pole, and the critical point $0$ of $f_2$ is sent to $z=\infty$ in the sphere, shown as the north pole.  The portion of the filled Julia set of the mating $F$ which corresponds to $J(f_1)$ (the rabbit) is shown in clear, and ``centered'' about the north pole. The portion  corresponding to $J(f_2)$ (the basilica) is shown in black on the front half of the sphere, and grey on the back half (to indicate that to see this, you are looking through $J(f_1)$).   However, due to the symmetry of the Julia sets of quadratic polynomials, this image is invariant under $180$ degree rotation about the vertical axis, hence the grey image in the back does not convey new information.  
Also, the fixed point $z=1$ (corresponding to the $\beta$-fixed points of $f_1, f_2$), is in the dead center of the image, in the front.  
Note reversing the order of mating, drawing the image of $1/3 \mate 1/7$, would have the effect of a $180$ degree rotation about the central horizontal axis (from $z=1$ to $z=-1$), and flipping the colors.

\medskip \noindent{\bf Self-mating.} The limb of the mandelbrot set enclosed by rays of angle $1/3, 2/3$ (see Figure~\ref{fig:mandelbrot}) is the only limb which is its own complex conjugate.  As such,  any PCF quadratic polynomial which is not in that limb can be mated with itself.  Such a mating clearly has extra symmetries.  The leftmost image in Figure~\ref{fig:rabmaterab} is the rabbit $1/7$ mated with itself.  We discuss self matings more in \S\ref{sec:selfmate}.

%---------------------------------------------------
\begin{figure}
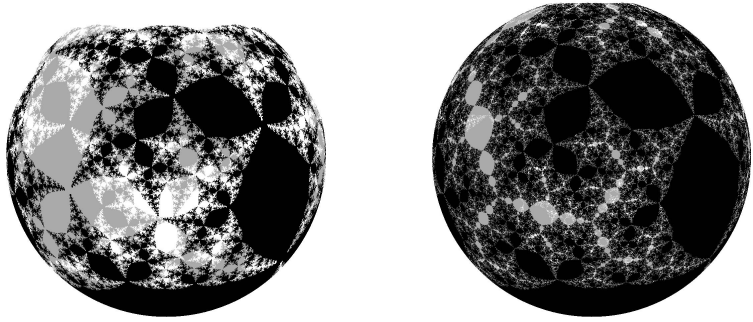

\begin{center}
   \drawfigrabmaterab  \hspace{0.5in}
      \drawfigrabMbasinrab
%   \drawfigrabMrabinrab
   \caption{
 \label{fig:rabmaterab}
Left: $1/7$ mate $1/7$, rabbit mate rabbit.
Right: $1/7$ mate $10/63$, i.e., replace each disk in the leftmost clear rabbit with a basillica.
%$1/7 \mate 74/511$, i.e., replace each disk in the leftmost clear rabbit with another rabbit.
}
\end{center}
\end{figure}
%---------------------------------------------------

\medskip \noindent {\bf Tuning.}  One simple way to make a mating more complicated is by tuning one of the quadratic polynomials.  The result shows up as you would expect.   In figure~\ref{fig:rabmaterab}, compare the rabbit mate rabbit on the left with the right figure, in which the clear rabbit has been tuned with a basilica.  We explore further expectations (and surprises) concerning tunings in \S\ref{sec:sequences}.

%===========================
\subsection{Shared Matings}
\label{sec:sharedmatings}
%===========================

One of the intriguing observations in the study of matings is that it can happen that two distinct pairs of PCF quadratic polynomials give rise to the same mating $F$.
If $f_1 \mate f_2 \cong F \cong f_3 \mate f_4$, and $f_1 \neq f_3$ or $f_2 \neq f_4$, then we call $F$ a {\em shared mating}. 

The simplest kind of shared mating is when $f_1 \mate f_2 \cong f_2 \mate f_1$. For example, the left side of Figure~\ref{fig:rabmateaero} illustrates such a shared mating of the rabbit ($1/7$) and aeroplane ($3/7$).
Of course, taking a shared mating and performing the same tuning on each quadratic polynomial will produce another shared mating, for example as on the right side of  Figure~\ref{fig:rabmateaero}.

%---------------------------------------------------
\begin{figure}
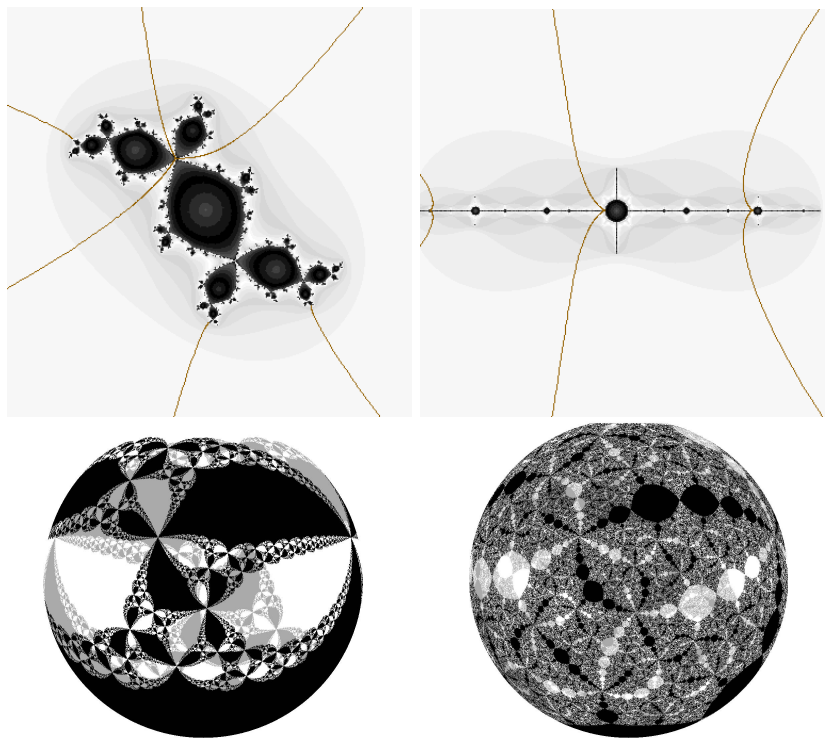

\begin{center}
    \drawfigJonesevenththreeseventh \ \drawfigJthreeseventhoneseventh
   \\   \drawfigrabmateaero
      \hspace{.5in}
      \drawfigbasinrabmateaero
   \caption{
 \label{fig:rabmateaero}
Upper left:  the rabbit, $1/7$; Upper right: the aeroplane, $3/7$, both shown with both sets of critical orbit rays ($1/7, 2/7, 4/7, 3/7, 6/7, 5/7$). 
Lower left: the shared mating the rabbit mate the aeroplane, $1/7$ mate $3/7$, equivalently, the aeroplane mate the rabbit.
Lower right: basilicas in the rabbit mate basilicas in the aeroplane, $10/63$ mate $28/63$.
}
\end{center}
\end{figure}
%---------------------------------------------------

Wittner (\cite{WittnerT}) studied this, and related, shared matings.

%===================================
\subsection{Space-filling curves and Latt\'{e}s mappings}
\label{sec:lattes}
%==================================

A very different example of a shared mating, discussed in detail in \cite{Mil-pasting}, is a \textit{Latt\'{e}s} map which can be realized as a mating  in four distinct ways:  

\bigskip 

$1/6 \mate 5/14 \ \ \cong \ \ 3/14 \mate 3/14 \ \ \cong \ \ 3/14 \mate 1/2 \ \ \cong \ \ 5/6 \mate 1/2.$

\bigskip

It is not known whether there is a bound on the number of ways in which a quadratic rational map can be realized as a mating. The quadratic polynomials involved above are: $f_{1/6} (z) = z^2 +i,$ a tip point on the rabbit limb; $f_{5/6}(z) = z^2-i$, the complex conjugate of $f_{1/6}$;  $f_{5/14}$, a tip point of the bulb on the basilica bulb corresponding to the rabbit; and $f_{1/2} (z) = z^2 - 2$, the real tip point of the basilica limb (the leftmost point in the mandelbrot set). 
The Julia set for each of $1/6, 5/14, 3/14, 1/2$ is a dendrite, hence has empty interior. For example, the Julia set of $f_{1/4}$ is a dendrite, shown in Figure~\ref{fig:lattessequence}.
Below is a characterization of when this occurs.

\begin{fact} \label{fact:evendenom}
Suppose $P_c$ is a PCF quadratic polynomial. Let $\theta_c = p/q$ be a reduced fraction.  TFAE:
\begin{enumerate}
\item $K_c$ has empty interior; 
\item $q$ is even;
\item $\theta_c$ is strictly pre-periodic under angle doubling.
\end{enumerate}
\end{fact}

Thus the mating of any two quadratic polynomials satisfying Fact~\ref{fact:evendenom} (including the shared mating above) has Julia set the entire Rieman sphere.
You can visualize such a mating as a space-filling curve on the sphere (each of the empty interior Julia sets is a curve which is pulled into becoming a space-filling curve). 
Further, since the  Julia set of $f_{1/2}$ is a line segment, any mating of the form $p/q \mate 1/2$ where $q$ is even will create a space-filling Peano curve.

Since the Julia set is the entire Riemann sphere, we cannot very well study such matings by drawing their Julia sets.  The harmonic measure supported on the Julia set is an object which deserves further study.
One could hope to learn something by examining the approximations to the Julia set drawn by the program Medusa in the steps of the algorithm converging to the mating.
See Figure~\ref{fig:lattesshared}.

%---------------------------------------------------
\begin{figure}
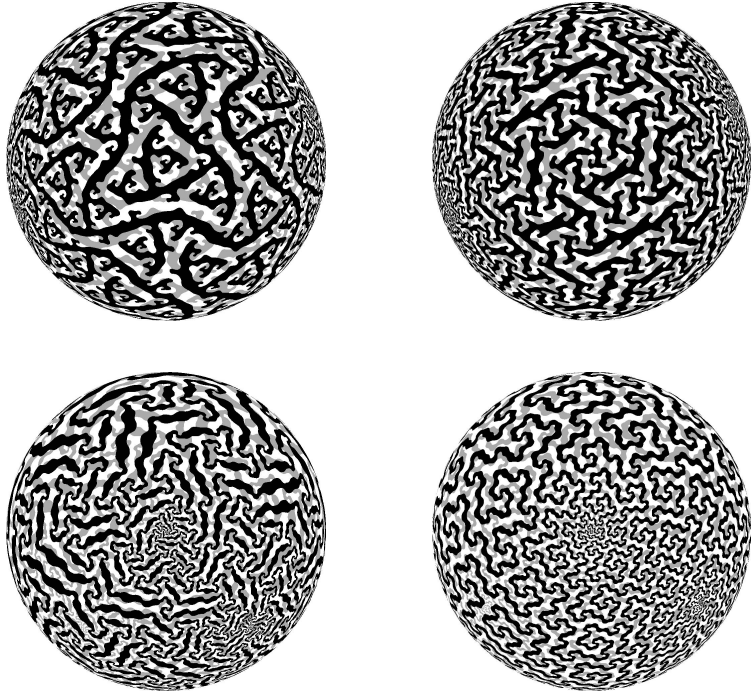

\begin{center}
   \drawfiglattesone
      \hspace{.5in}
      \drawfiglattestwo
 \vspace{.25in}
      \drawfiglattesthree
      \hspace{.5in}
      \drawfiglattesfour
   \caption{Each of the four images above illustrates a Medusa approximation $\mathcal{K}_{12}$ to the same shared Latt\'{e}s mating.  
   Upper left: $1/6$ mate $5/14$.  Upper right: $3/14$ mate $3/14$.  Lower left: $1/2$ mate $3/14$.  Lower right: $1/2$ mate $5/6$. (Note the two lower figures are mated in reverse order from the shared mating.  Just rotate the picture 180 degrees and exchange the colors to see the correct image).
 \label{fig:lattesshared}
}
\end{center}
\end{figure}
%---------------------------------------------------

%===================================
\subsection{Self Matings}
\label{sec:selfmate}
%===================================

Carston Peterson has observed that if $f$ is any PCF quadratic polynomial which is not in the $1/2$-limb of the Mandelbrot set (i.e., not in the unique limb which is its own complex conjugate), then
the following two rational maps are topologically conjugate:
\begin{enumerate}
\item start with $f \mate f$, then mod out by the obvious symmetry, and
\item $f \mate f_{1/2}$, where $f_{1/2}(z) = z^2 - 2$.
\end{enumerate}

This is because for $f_{1/2}$, the Julia set is a line segment, $[-2,2]$, and every external ray of angle $\theta$ has the same landing point as the ray of angle $1-\theta$ (the ray $0$ is horizontal and lands at $2$, the ray of angle $1/2$ is horizontal and lands at $-2$).  

For example, shown in Figure~\ref{fig:babyrabselfmate} is the Julia set of $f_{1/5}$, together with the Julia sets of both the self mating of $f_{1/5}$, and the mating of $1/5$ with $1/2$.  Since the Julia set of $1/2$ is simply a line segment, note in the figure how this simple segment is twisted to fill up all of the black.  

%---------------------------------------------------
\begin{figure}
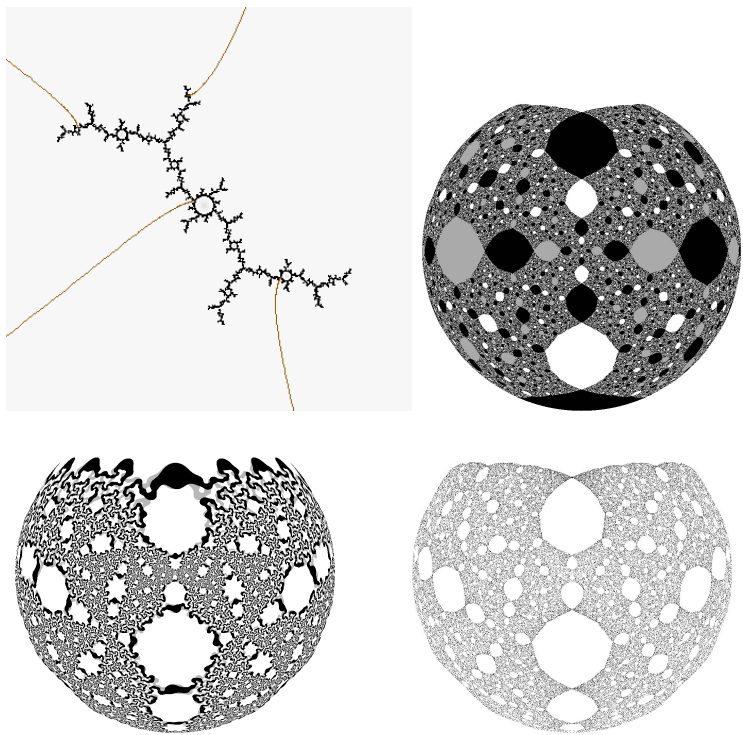

\begin{center}
 \drawfigJonefifth \ 
   \drawfigbabyrabmatebabyrab  
\\
      \drawfigbabyrabMlineapprox
       \hspace{0.35in}
      \drawfigbabyrabMline
   \caption{
 \label{fig:babyrabselfmate}
Upper Left: The Julia set of $f_{1/5}$, which is the center of largest baby Mandelbrot set off of the rabbit bulb, shown with critical orbit rays $1/5, 2/5, 4/5$. Upper Right: $1/5$ mate $1/5$.
Lower Right: $1/2$ mate $1/5$, i.e., mod out the upper figure by the obvious symmetry. 
Lower Left: an approximation $\mathcal{K}_{16}$, to $1/2$ mate $1/5$. The black is $1/2$, so shows the simple line twisting to fill up the alloted space.
}
\end{center}
\end{figure}
%---------------------------------------------------

%===================================
\subsection{Sequences of matings, and their limits}
\label{sec:sequences}
%==================================

One question about matings which has yielded an interested study is: If $f_1$ and $f_2$ are quadratic polynomials not in complex conjugate limbs, which are not PCF, when does a mating exist (assuming connected Julia sets)?  If $f_1$ and $f_2$ are hyperbolic, thus stable perturbations of hyperbolic PCF polynomials $g_1, g_2$, each with a super attracting periodic cycle, the mating exists as a deformation of the mating of $g_1, g_2$.   Several papers have appeared constructing matings between particular non-hyperbolic polynomials (see Ha\"{i}ssinksy and Tan Lei \cite{HTL},  Luo \cite{JLuo}, Yampolsky, Zakeri \cite{YZ-siegel}.) However, Epstein \cite{Epstein-QMC} has shown that mating does not extend continuously to the boundary of the hyperbolic component (in fact, the set of points in $\partial {M} \times \partial {M}$ where there is no continouous extension is dense).  Epstein's theorem is that an obstruction to continuously extending this map to a mating between the two root points of the hyperbolic components occurs whenever in the mating $g_1 \mate g_2$, the immediate basins of the superattracting cycles of $g_1, g_2$ touch along a distinguished repelling cycle (excluding $g_i(z) = z^2$).  For example, this occurs in the mating of the rabbit and the aeroplane, Figure~\ref{fig:rabmateaero}.  That this is a shared mating is an additional coincidence, not needed for Epstein's theorem.  

We can  use Medusa to see a different type of example of why mating as a map from ${M} \times {M}$ to the space of quadratic rational maps is not continuous.   We examine a few convergent sequences of quadratic polynomials, 
$\theta_m, \omega_m \to \theta, \omega,$ as $m\to \infty$,
%$p_m / q_m \to p/q, r_m / s_m \to r/s$, 
such that the mating 
%of $p_m / q_m$ with $r_m / s_m$ 
$\theta_m \mate \omega_m$
exists for every $m$, but 
%$p/q \mate r/s$ 
$\theta \mate \omega$
either does not exist, or is not the limit of 
$\theta_m \mate \omega_m$.
%$p_m / q_m \mate r_m / s_m$.  

Below are some simple examples of sequences with no limit, or the wrong limits.
\begin{enumerate}
\item First consider $\theta_m = \omega_m = \frac{1}{2^m - 1}$, so $\theta = \omega = 0$.  Note $0$ corresponds to $z \mapsto z^2$, so $\theta \mate \omega = 0 \mate 0$ is just $z \mapsto z^2$, with Julia set the circle.  However, Medusa output suggests that the Julia set of $\theta_m \mate \omega_m$ is much more complicated than the unit disk.  The leftmost image in Figure~\ref{fig:sequences} shows the Julia set of $1/255 \mate 1/255$ (recall Figure~\ref{fig:rabmaterab} shows the first element of the sequence, $1/7 \mate 1/7$).
\item A similar example is given by $\theta_m =  \frac{1}{2^m - 1}, \omega_m = \frac{2^{m-1}-1}{2^m-1}$, so $\theta=0$ and $\omega = 1/2$.  Note $0 \mate 1/2$ is just $1/2$, i.e., $z \mapsto z^2-2$, with Julia set $[-2,2]$.
As in the previous example, Medusa output shows $\theta_m \mate \omega_m$ is quite complicated.  The center of Figure~\ref{fig:sequences} shows $1/511 \mate 255/511$ (and Figure~\ref{fig:rabmateaero} shows the first element of the sequence, the rabbit mate the aeroplane). 
\item Finally, we examine $\theta_m = \omega_m = \frac{(2^{2m}-1)(2/3) - 1 }{2^{2m+3}-1}$ (i.e., the sequence $9/31, 41/127, 169/511, \ldots$, of angles of the upper ray landing at the root point of the bulbs proceeding from the rabbit to the basilica), hence $\theta = \omega = 1/3$.  Since $f_{1/3}(z)=z^2-1$ is the basilica, it is its own complex conjugate, and its self mating does not exist.  The rightmost image in Figure~\ref{fig:sequences} is the Julia set of $169/511$ mated with itself.
\end{enumerate}
% 9/31, 41/127, 169/511....

%---------------------------------------------------

\begin{figure}
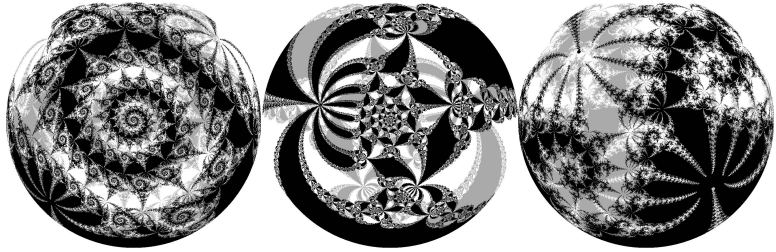

\begin{center}
   \drawfigseqone \drawfigseqtwo \drawfigseqthree
   \caption{
 \label{fig:sequences} Medusa images of the Julia sets of the following matings:
Left: $1/511$ mate $1/511$.
Center: $1/511$ mate $255/511$.
Right: $169/511$ mate $169/511$.
\textit{I should elaborate...}
}
\end{center}
\end{figure}

\bibliographystyle{alpha}
\bibliography{medusa}

 \end{document}